\documentclass[11pt,a4paper]{article}

\usepackage[english]{babel} 
\usepackage[T1]{fontenc}
\usepackage{lmodern}
\usepackage{textcomp}
\usepackage{amssymb} 
\usepackage{amsthm}
\usepackage{mathtools}
\usepackage{mathrsfs}  
\usepackage{enumitem}
\usepackage[pdftex]{graphicx}
\usepackage[export]{adjustbox}
\usepackage{color}
\usepackage{subfigure}
\usepackage[arrow, matrix, curve]{xy}
\usepackage{lmodern}
\usepackage[utf8]{inputenc}
\usepackage{url}
\usepackage{hyperref}
\usepackage{float}
\usepackage{chngcntr}
\usepackage{listings}
\usepackage{appendix}
\usepackage{pdfpages}
\usepackage{algorithm}
\usepackage{algorithmic}
\usepackage{geometry}
\usepackage{bbm}
\usepackage{bbold}
\usepackage{graphicx}
\usepackage{xcolor}
\usepackage{xargs} % use more than one optional parameter in a new command
\usepackage[colorinlistoftodos,textsize=small]{todonotes}
\usepackage[backend=bibtex, style=numeric, citestyle=numeric, hyperref=true]{biblatex}
\nocite{*} 
\usepackage{todonotes}

\definecolor{codegreen}{rgb}{0,0.6,0}
\definecolor{codegray}{rgb}{0.5,0.5,0.5}
\definecolor{codepurple}{rgb}{0.58,0,0.82}
\lstdefinestyle{mystyle}{
	basicstyle=\tiny,
	commentstyle=\color{codegreen},
	keywordstyle=\color{blue},
	numberstyle=\tiny\color{codegray},
	stringstyle=\color{codepurple},
	breakatwhitespace=false,         
	breaklines=true,                 
	captionpos=b,                    
	keepspaces=true,                 
	numbers=left,                    
	numbersep=5pt,                  
	showspaces=false,                
	showstringspaces=false,
	showtabs=false,                  
	tabsize=2
}

\newcommand{\R}{\mathbb{R}}

\newcommand{\pa}{\partial}

\newcommand{\ve}{\varepsilon}

\newcommand{\vp}{\varphi}

\newcommand{\md}{\,\mathrm{d}}

\newcommand{\supp}{\operatorname{supp}}

\newtheorem{theorem}{Theorem}
\newtheorem{lemma}[theorem]{Lemma}
\newtheorem{cor}[theorem]{Corollary}
\theoremstyle{definition}
  
\newtheorem{rem}{Remark}
\newtheorem{assumption}{Assumption}

\newcommandx{\unsure}[2][1=]{\todo[linecolor=red,backgroundcolor=red!25,bordercolor=red,#1]{#2}}
\newcommandx{\change}[2][1=]{\todo[linecolor=blue,backgroundcolor=blue!25,bordercolor=blue,#1]{#2}}
\newcommandx{\info}[2][1=]{\todo[linecolor=green,backgroundcolor=green!25,bordercolor=green,#1]{#2}}
\newcommandx{\improvement}[2][1=]{\todo[linecolor=yellow,backgroundcolor=yellow!25,bordercolor=yellow,#1]{#2}}

\newcommandx{\biblio}[2][1=]{\todo[linecolor=blue,backgroundcolor=magenta!25,bordercolor=blue,#1]{#2}}
\newcommandx{\laura}[2][1=]{\todo[linecolor=violet,backgroundcolor=violet!25,bordercolor=violet,#1]{#2}}

\definecolor{darkblue}{rgb}{0,0,0.7}

\begin{document}
	\title{Boltzmann mean-field game model for knowledge growth: limits to learning and general utilities}
	\author{Martin Burger\thanks{Computational Imaging Group and Helmholtz Imaging, Deutsches Elektronen-Synchrotron (DESY), Notkestr. 85, 22607 Hamburg and Fachbereich Mathematik, Universit\"at Hamburg. \tt{martin.burger@desy.de}} \and L. Kanzler\thanks{CEREMADE, UMR 7534, Universite Paris–Dauphine, Place du Marechal de Lattre de Tassigny, 75775 Paris Cedex 16, France. {\tt laura.kanzler@dauphine.psl.eu}} \and M.T. Wolfram\thanks{Mathematics Institute, University of Warwick, Coventry CV4 7AL, United Kingdom \tt{M.Wolfram@warwick.ac.uk}}}
	\date{\vspace{-5ex}}
	\maketitle
	
	\begin{abstract}
    In this paper we investigate a generalisation of a Boltzmann mean field game (BMFG) for knowledge growth, originally introduced by the economists Lucas and Moll \cite{ML}. In BMFG the evolution of the agent density with respect to their knowledge level is described by a Boltzmann equation. Agents increase their knowledge through binary interactions with others; their increase is modulated by the interaction and learning rate: Agents with similar knowledge learn more in encounters, while agents with very different levels benefit less from learning interactions. The optimal fraction of time spent on learning is calculated by a Hamilton-Jacobi-Bellman equation, resulting in a highly nonlinear forward-backward in time PDE system.\\
    The structure of solutions to the Boltzmann and Hamilton-Jacobi-Bellman equation depends strongly on the learning rate in the Boltzmann collision kernel as well as the utility function in the Hamilton-Jacobi-Bellman equation. In this paper we investigate the monotonicity behaviour of solutions for different learning and utility functions, show existence of solutions and investigate how they impact the existence of so-called balanced growth path solutions, that relate to exponential growth of the overall economy. Furthermore we corroborate and illustrate our analytical results with computational experiments. 
	\end{abstract}
	
	
	\vspace{0.5cm}
	
	
	\vspace{0.5cm}
	
	\section{Introduction}
	
	In this paper we investigate a Boltzmann mean field game (BMFG) model for knowledge growth in large societies, originally introduced by Lucas and Moll \cite{ML}. In this model agents are characterised by their knowledge level $z$ (in $\Omega = \mathbb{R}_+$ or $\Omega = [0, \bar{z}]$ with $\bar{z} \in \mathbb{R}^+$), and the time they spend on learning/increasing their knowledge level. Agents learn through interactions with others, and determine the best amount of learning time solving an optimal control problem. This leads to a coupling of a Boltzmann type equation, describing the interactions of agents, to a mean field game; hence a Boltzmann mean-field game. There has been a significant interest in the analysis and simulation of BMFGs in the last years, see \cite{BLW, BLW2, PR, PRV}. \\ 
	Boltzmann type equations have been used successfully to describe the interactions of large interacting agent systems in socio-economic sciences, for example in wealth distribution \cite{CC2000, pareschi2014wealth, MT2008}, price formation \cite{BCMW2013}, opinion formation \cite{T2006,During:strongleaders} or ranking in sports \cite{JJ2015}; see \cite{PT2013} for a more general introduction to kinetic models. Mean-field games were originally introduced by Lasry and Lions \cite{LL2007} and Caines et al \cite{CHM2006}; an introduction from the PDE perspective can be found in \cite{CP2020}. Classical mean-field games correspond to the Nash equilibrium configuration of a differential game of infinitely many players. It consists of a forward Fokker-Planck equation describing the evolution of the player distribution, which is coupled to a backward Hamilton-Jacobi-Bellman equation for the value function of a single player. For certain types of cost the coupled MFG system can be interpreted as a parabolic optimal control problem, so called variational or potential mean field games, see for example \cite{Benamou2017}. Potential mean-field games have a similar structure as the Benamou-Brenier formulation for optimal transport, see \cite{BB2000}. We will see in Section \ref{s:lmfg} that the proposed BMFG converges (formally) to a specific type of potential MFG in a suitable scaling limit; establishing a first formal connection between BMFGs and MFGs.\\
	Note that the form of the Boltzmann equation in BMFGs resembles to some extend the one from the Smoluchowski coagulation equations, see \cite{S1916}. However, the equation's fundamental structure is different since in the case of the Smoluchowski coagulation models symmetric and homogeneous kernels are considered. In our case we either consider an homogeneous, with index 0, but not symmetric kernel or a symmetric kernel without any homogeneity property. Moreover, the aforesaid coagulation equation models the fusion of a particle with size $y$ with one of size $x-y$, resulting in the gain of a particle of size $x$. Hence, the sum of the states remains constant within each binary interaction, which causes together with the symmetry property of the kernel that the mean value is conserved. This is not the case in our model since during a binary interaction the individual with smaller knowledge increases it, while the on of the interaction-partner remains unchanged, which causes increase of the mean with time. Therefore, analytical techniques and results as developed in \cite{FL2005, NV2013} can not be applied.

	\subsection{The Boltzmann mean-field game}
	Lucas and Moll proposed several generalisations of the proposed basic BMFG model in \cite{ML}. In this paper we will investigate two of them - how limits to learning as well as the change of the utility function effects the qualitative behaviour of solutions. \\
	
	We start by considering limits to learning and recall that agents are characterised by their knowledge level $z \in \Omega$. In the original BMFG model, agents learned from others with a higher knowledge level. Now we assume that an agent with knowledge level $z$ interacts with another individual of lower knowledge level $y$ with a certain rate
	$$
	k(z,y),
	$$
	which is assumed to be decreasing as the difference between the knowledge levels increase. Hence agents learn less, if the difference in the knowledge level is too big. We will focus, among the more general case, on two special cases of the interaction function, a polynomial kernel 
	\begin{align}\label{e:kpoly}
		k(z,y) = \delta + (1-\delta) \left(\frac{y}{z}\right)^{\kappa} 
	\end{align}
	for $\delta \in (0,1)$ and $\kappa > 0$ and an exponential one
	\begin{align}
		\label{e:kexp}
		k(z,y) = \mu e^{-\kappa \lvert z - y \rvert} \text{ with } \mu, \kappa >0.
	\end{align}
	In case of the polynomial kernel, agents always learn at the minimum rate $\delta$, no matter how big the knowledge gap is. This is not the case for the second choice \eqref{e:kexp}, in which the knowledge gain in interactions approaches zero exponentially fast as the difference in knowledge grows. We will see that the different decay behaviour of the kernel $k$ can have an influence on the long time behaviour. In particular, we are able to show that balanced growth path (BGP) solutions might degenerate to constant states in the case of the exponential kernel \eqref{e:kexp}. We will discuss BGPs in more detail in Section \ref{s:bgp}.\\
	Let $f = f(z,t)$ denote the distribution of agents with respect to the knowledge level. Then the evolution of agents can be described by a Boltzmann equation of the form
	\begin{equation}\label{BoE}
		\pa_t f(z,t) = f(z,t) \int_0^z \alpha\left(s(y,t)\right)k(z,y)f(y,t) \, \md y-\alpha\left(s(z,t)\right) f(z,t) \int_z^{\infty} k(y,z)f(y,t) \, \md y,
	\end{equation}
	where $\alpha=\alpha(s(z,t)): [0,1] \to [0,1]$ denotes the learning rate. 
	The function $s=s(z,t) \in [0,1]$ corresponds to the fraction of time an individual with knowledge $z$ spends on learning. The first term on the right-hand-side of \eqref{BoE} are the gains due to interactions of an agent with knowledge level $z$ with others having a lower knowledge level. The second term is the loss due to interactions of agents with knowledge level $z$ with others having a higher knowledge level (at rate $\alpha(s(z,t))$). \\
	Agents determine what fraction of their time they should spend on learning, modelled by the function $s(z,t)$, or working, that is $1-s(z,t)$, by maximising their productivity. This corresponds to an optimal control problem, resulting in a coupling to a Hamilton-Jacobi-Bellman equation. Let $V = V(z,t)$ denote the value function, which corresponds to the outcome that an agent with initial knowledge level $z$ can expect when optimizing over the time horizon $[t, \infty)$. The value function $V$ satisfies the following Hamilton-Jacobi-Bellman equation:
	\begin{equation}\label{BeE}
		\pa_t V(z,t) - rV(z,t) = -\max_{s \in \mathcal{S}}{\left[U((1-s)z) + \alpha(s) \int_z^{\infty} \left(V(y,t)-V(z,t)\right)f(y,t)k(y,z) \, \md y \right]},
	\end{equation}
	where  $p= (1-s)z$ is the individual productivity, i.e. productivity is proportional to the knowledge level times the spent actually working, and the maximum is taken over the set
	$$
	\mathcal{S}:=\left\{s: \, [0,\infty) \times [0,T] \to [0,1]\right\},
	$$
	of possible time-fractions spent on learning. The function $U$ is the so-called utility, which relates the individual productivity $y$ to the expected gain. Possible choices include the linear utility
	\begin{align}\label{e:ulin}
		U(p) = p
	\end{align}
	or the logarithmic utility
	\begin{align}\label{e:ulog}
		U(p) = \ln(p).
	\end{align} 
	The linear utility  \eqref{e:ulin} and the logarithmic utility \eqref{e:ulog} correspond to the limits, $\zeta =0$ and $\zeta=1$, of the class of isoelastic utility functions
	\begin{align}
		\label{e:ucrra}
		U(p) = 
		\frac{p^{1-\zeta}}{1-\zeta} &\text{ with } \zeta \in (0,1).
	\end{align} 
	Note that the isoelastic utility \eqref{e:ucrra} is often stated with an additional $-1$ in the nominator, that is $U(p) =\frac{p^{1-\zeta}-1}{1-\zeta}$. However, this additive constant does not change the optimal decision. We prefer to work with the positive version of the utility and will consider utility functions of the form \eqref{e:ucrra} throughout this paper.
	Function \eqref{e:ucrra} is a constant relative risk aversion (CRRA) function, which means that the quantity $\eta = -p \frac{U''(p)}{U'(p)}$ is constant. Note that the logarithmic utility \eqref{e:ulog} relates to high risk aversion, while the linear utility  is risk neutral that is $\eta = 0$. \\
	Equation \eqref{BeE} is supplemented with a terminal condition of the form
	\begin{align*}
		V(z,T) = 0,
	\end{align*}
	while the Boltzmann equation is initialised at time $t=0$, in particular $f(z,0) = f_I(z)$.\\
	In summary we obtain the fully coupled system:
	\begin{subequations}\label{bmfg}
		\begin{align}
			\pa_t f(z,t) &= f(z,t) \int_0^z \alpha\left(s(y,t)\right)k(z,y)f(y,t) \, \md y\label{bmfgf}\\
			&\quad -\alpha\left(s(z,t)\right) f(z,t) \int_z^{\infty} k(y,z)f(y,t) \, \md y,\nonumber \\
			\pa_t V(z,t) - rV(z,t) &= -\max_{s \in \mathcal{S}}{\left[ U(p) + \alpha(s) \int_z^{\infty} \left(V(y,t)-V(z,t)\right)f(y,t)k(y,z) \, \md y \right]},\label{bmfgV}\\
			S(z,t) &= \arg \max_{s \in \mathcal{S}} {\left[ U(p) + \alpha(s) \int_z^{\infty} \left(V(y,t)-V(z,t)\right)f(y,t)k(y,z) \, \md y \right]}, \\
			f(z,0) &= f_I(z)\\
			V(z,T) &= 0,
		\end{align}
	\end{subequations}
	where $f_I$ is the initial distribution of agents.\\
	
	We conclude this subsection by discussing the notion of balanced growth path (BGP) solutions for system \eqref{bmfg}.  Assume there exists a constant $\gamma \in \mathbb{R}^+$ and define $x = z e^{-\gamma t}$ as well as the functions
	\begin{align}\label{e:rescalbgp}
		&f(z,t) =  e^{-\gamma t}\phi(ze^{-\gamma t}), ~~V(z,t)= e^{\gamma t}v(ze^{-\gamma t}) \,\text{ and }\, s(z,t)= \sigma (ze^{-\gamma t}).
	\end{align}
	Then the Boltzmann mean field game \eqref{bmfg} in the new variables $(\phi, v, \sigma) = (\phi(x), v(x), \sigma(x))$ becomes  
	\begin{subequations}\label{bgp}
		\begin{align}
			-\gamma \phi(x) -\gamma x \phi'(x) &= \phi(x)\int_0^x\alpha(\sigma(y)) \tilde{k}(t,x,y) \phi(y)\,\md y - \alpha(\sigma(x))\phi(x)\int_x^\infty \tilde{k}(t,y,x)\phi(y)\,\md y\label{e:phi}\\
			(r-\gamma)v(x)+\gamma x v'(x) &=  \max_{\sigma \in \Xi}\left[U(p)+\alpha(\sigma)\int_x^\infty[v(y)-v(x)]\tilde{k}(t,y,x) \phi(y)\,\md y \right] \label{e:v}\\
			\Sigma(x) &= \arg \max_{\sigma \in \Xi}\left\{U(p)+\alpha(\sigma)\int_x^\infty[v(y)-v(x)]\tilde{k}(t,y,x)\phi(y)\, \md y \right\}\label{e:S}
		\end{align}
	\end{subequations}
	where $\Xi := \left\{\sigma:\, [0,\infty) \to [0,1]\right\}$ and where for the polynomial learning kernel we have
	\begin{align*}
		\tilde{k}(t,y,x) = k(y,x) = \delta + (1-\delta) \left(\frac{y}{x}\right)^{\kappa} ,
	\end{align*}
	while the exponential learning kernel inherits a time-dependency in the BGP variables
	\begin{align*}
		\tilde{k}(t,y,x) = \mu e^{-\kappa e^{\gamma t} \lvert x - y \rvert}.
	\end{align*}\\
	Then the overall productivity of an economy, defined as
	\begin{align}\label{e:overallproductivity}
		Y(t) = \int_0^{\infty} U((1-S(z,t))z)f(z,t) \,dz,
	\end{align}
	can be written in the new variables as 
	$$Y(t) =  \int_{0}^{\infty}U\left((1-\Sigma(x))e^{\gamma t} x\right) \phi(x) dx,
	$$
where we defined the maximizer in BGP variables by $\Sigma \left(ze^{-\gamma t}\right) = S(z,t)$. So if a rescaling of type \eqref{e:rescalbgp} with $\gamma > 0$ exists, then the overall productivity	grows  exponentially in time. Economists relate \eqref{e:overallproductivity} to the gross domestic product (GDP) of an economy, which is known to grow exponentially for most developed countries (in the long time run). This is why economists are particularly interested in the existence of such a rescaling and the respective BGP solutions $(\phi, v, \sigma)$. In the original model \eqref{bmfg} with $k \equiv 1$ a necessary assumption for the existence of such a rescaling is that the initial cumulative distribution function of agents with respect to their knowledge level has a Pareto tail,  meaning that there is enough high knowledge so that the individuals with lower knowledge levels can learn, see \cite{BLW2}. Knowledge diffusion, as considered in 
\cite{PR, PRV}, also ensures the existence of BGPs. In this case the diffusivity enters as a multiplicative constant in the knowledge growth. Note that the growth parameter $\gamma$ relates to the wave speed of travelling wave solutions to the problem in logarithmic variables, see \cite{PR} for further details. The magnitude of the growth parameter $\gamma$ is influenced by the interaction function $k$ (as individuals learn less if the difference in their knowledge levels is too big). We will present first results in Section \ref{s:bgp}.

	\subsection{Our contribution}
	\noindent In this paper we present analytical and computational results investigating the behaviour of system \eqref{bmfg} for different interaction kernels $k$ and utility functions $U$. We briefly recall the main analytical results of system \eqref{bmfg} in the case of the trivial learning kernel $k \equiv 1$ and linear utility \eqref{e:ulin}, see Burger et al. \cite{BLW, BLW2}: 
	\begin{itemize}[nosep]
		\item The value function $V$ is a non-decreasing function of the knowledge level $z$ for all times $t>0$.
		\item The optimal learning time fraction $S$ is a non-increasing function of the knowledge level $z$ for all times $t>0$.
		\item Pareto tails of the initial cumulative distribution function are preserved in time.
		\item The existence of a Pareto tail ensures the existence of BGP solutions.
	\end{itemize}

	\vspace{0.3cm}
	We will show that general learning kernels $k(x,y)$ and utility functions $U(y)$ have impact on the monotonicity behaviour of solutions as well as the existence of BGPs. In particular we show:
	\begin{itemize}[nosep]
		\item \emph{Existence and uniqueness of a solution} to \eqref{bmfg} for small time intervals is valid for all bounded learning kernels (especially for \eqref{e:kpoly} and \eqref{e:kexp}) and for all isoelastic utility functions \eqref{e:ucrra} as well as for the linear utility function \eqref{e:ulin}, see Theorem \ref{t:existence} in Section \ref{s:analysis_coupled_systems}.
		\item The \emph{value function $V(\cdot,t)$, $t \in [0,T)$, is non-decreasing} for all choices of learning kernels and utility functions considered in this paper, see Theorem \ref{t:Vmon} in Section \ref{s:bellman}.
		\item The \emph{value function $V(z,t)$ is non-negative} for all $(z,t) \in \R_+ \times \R_+$, for isoelastic utility functions \eqref{e:ucrra} as well as for the linear utility function \eqref{e:ulin}, see Lemma \ref{l:Vpos}. However, it can obtain negative values for small $z$ in case of a logarithmic utility \eqref{e:ulog}, see Lemma \ref{l:Vposln}.
		\item The optimal learning time fraction \emph{$S(\cdot,t)$ is non-increasing} for all times $t \in [0,T)$ for isoelastic and linear utilities in the following situations (see Corollary \ref{c:Smon}, Section \ref{s:bellman})
		\begin{itemize}
		\item In case of the polynomial learning kernel \eqref{e:kpoly} the involved parameters have to satisfy $1-\zeta \geq \kappa (1-\delta)$. 
		\item In case of the exponential learning kernel \eqref{e:kexp} there are no restrictions on the parameters involved. However, the results only holds true in the interval $z \in \left[0,\frac{1-\zeta}{\kappa}\right]$.
		\end{itemize}
		Numerical simulations show that \emph{non-monotonous learning behaviour} can be observed for the logarithmic utility, see Figure \ref{f:log_v} in Section \ref{s:numerics}.
		\item \emph{Pareto tails} of the initial cumulative distribution function are preserved in time, but the tail index can get arbitrarily close to zero in case of the exponential learning kernels, see Section \ref{s:bgp}. We note that the existence of such BGP solutions remains an open problem, but we believe that the proof should follow the lines of  \cite{BLW2}.
		\item A formal localisation in the learning kernel leads to a new class of \emph{local mean-field game models}, which shares strong structural similarities with potential mean field games, see Section \ref{s:lmfg}.
	\end{itemize}
	
	\vspace*{1em}
	
	\noindent These results are organised in the following way throughout this article: Section \ref{s:analysis} is dedicated to the analytical investigation of \eqref{bmfg}, where we start by investigating the decoupled equations first. Section \ref{s:boltzmann} focuses on the analysis of the Boltzmann equation \eqref{bmfgf} for a given learning function $\alpha$, while Section \ref{s:bellman} analyses the Hamilton-Jacobi-Bellman equation \eqref{bmfgV} for different interaction kernels $k$ and utility functions $U$. Existence of solutions to the full system is presented in Section \ref{s:analysis_coupled_systems}. Section \ref{s:bgp} focuses on BGP solutions, before concluding with a discussion of a local mean field games model in Section \ref{s:lmfg}. This formal connection between BMFGs and MFGs is derived by considering the formal limit of a localised interaction kernel $k$. We finish by illustrating the behaviour of solutions with various computational experiments in Section \ref{s:numerics}.
	
	\subsection{Notation and assumptions}
	
	Throughout this paper we make the following assumptions (unless stated otherwise).
	
	\begin{assumption}[Assumptions on the learning kernel $k$]
		\label{a:k} The learning function $k \in C^1$ is non-increasing in the difference between knowledge levels, in particular $k(z,\cdot) \in C^1(\R_+)$ for all $z \in \R_+$ and
		\begin{align}
			\pa_yk(z,y) \geq 0, \quad \text{for } z \geq y.
		\end{align}
	\end{assumption}
	
	\begin{assumption}[Assumptions on the utility $U$]
		\label{a:U}
		Non-linear utility functions $U:\, \R_+ \to \R$, \newline $U \in C^2\big((0,\infty)\big)$ are assumed to be increasing and concave, that is 
		\begin{align*}
				U' > 0 \text{ on } [0,\infty), \text{ and }\, U'' < 0 \text{ on } (0,\infty), \,\, U'' \leq 0 \text{ on } [0,\infty).
		\end{align*}
	\end{assumption}
	
	\begin{assumption}[Assumptions on the learning rate $\alpha$]
		\label{a:alpha}
		The learning rate $\alpha:\, \R_+ \to \R_+$, \newline $\alpha \in C^2\big([0,\infty)\big)$ satisfies
		\begin{align*}
			\alpha: \R_+ \to \R_+,\, \alpha \in C^{\infty}\left([0,1]\right), \, \alpha(0)=0, \, \alpha'(0)=\infty,\,  \alpha''<0, \text{ and } \alpha'>0,
		\end{align*}
	\end{assumption}
	\begin{assumption}\label{a:fI}
		The initial condition for the distribution function $f$ of individuals fulfils
		\begin{align*}
			f(z,t=0) = f_I(z), \, z \in \R_+, \,\text{ with }\, f_I \in L^1(\R_+),\,\, \int_0^{\infty} f_I(z) \,  \md z = 1, \quad f_I(z)\geq 0, \, \forall z \geq 0.
		\end{align*}
	\end{assumption}
	
		Note that the considered polynomial and exponential learning kernel \eqref{e:kpoly} and \eqref{e:kexp}, as well as the utility functions \eqref{e:ulin}, \eqref{e:ulog}, \eqref{e:ucrra} satisfy Assumption \ref{a:k} and \ref{a:U}.
		
	\section{Analysis of the Boltzmann mean-field game}\label{s:analysis}
	
	This section focuses on the analytical investigation of the BMFG system \eqref{bmfg}. In Section \ref{s:boltzmann} we derive properties of the Boltzmann equation for a given learning function; its main result is a global in time $L^1$ existence and uniqueness of the solution. We proceed with Section \ref{s:bellman}, in which we investigate the Hamilton-Jacobi-Bellman equation for a given agent distribution $f$. These results are two-fold: On the one hand, we are able to show an existence and uniqueness result of a $L^\infty$ solution together with a stability estimate. On the other hand, we were able to show qualitative properties of $V$ and $S$. We then use these results in Section \ref{s:analysis_coupled_systems}, to show local existence and uniqueness of a solution to the fully coupled system \eqref{bmfg}.
	
	\subsection{Investigation of the Boltzmann equation}\label{s:boltzmann}
	
	We start analyzing the Boltzmann equation 
	\begin{equation}\label{BoE2}
		\begin{split}
			\pa_t f(z,t) &= f(z,t) \int_0^z \alpha(y,t) k(z,y)f(y,t) \, \md y-\alpha(z,t) f(z,t) \int_z^{\infty} k(y,z)f(y,t) \, \md y \\
			&=: G(f,f)(z) - L(f,f)(z)= Q(f,f)(z),
		\end{split}
	\end{equation}
	subject to the initial condition $f(z,t=0)=f_I(z)$, fulfilling Assumption \ref{a:fI} for a given learning function, i.e. $\alpha =\alpha (z,t) \in L^{\infty}\Big(\R_+^2\Big)$. 
	Note that boundedness of $\alpha$ holds if $s:\, \R_+\times [0,T] \to [0,1]$ and if Assumption \ref{a:alpha} is fulfilled, which implies an upper bound $\bar{\alpha} = \alpha(1) \geq \alpha(s(z,t))$, for all $z \in \R_+$ and $t<T \in \R_+$. 
	\begin{rem}
	    Note that the Boltzmann equation \eqref{BoE2} can be investigated considering probability density functions $f_I,f \in \mathcal{P}(\R_+)$ as initial condition, in which case more careful consideration in the definition of the integral operator in \eqref{BoE2} is needed. 
	\end{rem}
	In what follows we will also use the \emph{cumulative distribution function}
	\begin{align}\label{F}
		F(Z,t) = \int_0^Z f(z,t) \,  \md z.
	\end{align}
	From \eqref{BoE} we can easily deduce the following time evolution of $F$
	\begin{align}\label{BoEF}
		\pa_tF(Z,t) = \int_0^Z f(z,t) \int_0^z \alpha(y,t) k(z,y) f(y,t) \, \md y \,\md z - \int_0^Z \alpha(z,t) f(z,t) \int_z^{\infty} k(y,z) f(y,t)  \, \md y \, \md z.
	\end{align}
	Note that in the case $k \equiv 1$ equation \eqref{BoEF} can be transformed using the function $G = 1-F$, which satisfies
    \begin{align}\label{e:G}
        \partial_t G = G(1-G).
    \end{align}
    The viscous version of \eqref{e:G} exhibits travelling wave solution, which relates to BGP solutions, see \cite{BLW2, PR} for more details. 

	\paragraph{Conservation laws and properties of the collision operator:}
	
	Multiplying the right-hand-side of \eqref{BoE2} with a test-function $\vp$ and integrating over the state space $(0,\infty)$ we obtain the \emph{weak formulation of the collision operator
	\begin{align}\label{BoE2weak}
		\int_0^{\infty} Q(f,f)(z)\vp(z) \, \md z = \int_0^{\infty} \int_0^z f(z,t) f(y,t) k(z,y) \alpha(y,t) \left(\vp(z)-\vp(y)\right) \, \md y \, \md z.
	\end{align}
	}If we set $\vp \equiv 1$ in \eqref{BoE2weak}, we see that the total mass is conserved, that is
	\begin{align}\label{masscons}
		\frac{d}{dt} \int_0^{\infty} f(z,t) \, \md z  = 0, \,\, \text{and therefore} \, \int_0^{\infty} f(z,t) \, \md z =1 \text{ for all times } t\geq0.
	\end{align}
	This is consistent with the modelling assumption that no individuals are gained or lost; only a change in their knowledge level occurs. Moreover, the \emph{mean knowledge level} of the population is defined as
	\begin{align}\label{mean}
		m(t) := \int_0^{\infty} z f(z,t) \,  \md z,
	\end{align}
	which one finds to be monotonically increasing in time. Indeed, setting  $\vp(z)=z$ for all $z \in \R_+$ in \eqref{BoE2weak} we obtain that 
	\begin{align}\label{meanmon}
		\frac{d}{dt}m(t) := \int_0^{\infty} \int_0^z f(z,t) f(y,t) \alpha(y,t) k(z,y) (z-y) \, \md y \, \md z \geq 0. 
	\end{align}
	In the weak formulation of the collision operator the \emph{asymmetry of the individual's learning interactions}, characteristic for this dynamics, can be seen clearly. This is the crucial property which causes the mean knowledge level to be non-decreasing. Moreover, one notices immediately, that the above expression on the right-hand-side can only vanish, if we have $f(z,t) = \delta_{z_*}$ for some $z_* \in (0,\infty)$, meaning that the mean knowledge level will always increase until all knowledge is concentrated at one level. Moreover, a special case of the following Lemma ensures that no blow-up in finite time can occur for the first moment. More generally we have	
	\begin{lemma}\label{l:weightbound}
		Let the initial data $f_I$ fulfill \eqref{a:fI} as well as
		$$
			\int_0^{\infty} w(z) f_I(z) \, \md z < \infty,
		$$
		with a measurable, positive non-decreasing function $w:\, \R_+ \to \R_+$ and $\alpha \in L^{\infty}(\R^2_+)$. Moreover, assume $k(z,y) \leq \bar{k}$ for all $y<z$. Then the integral of a solution $f=f(z,t)$ to \eqref{BoE2} weighted by $w$ is non-decreasing with respect to time and bounded by
		$$
			\int_0^{\infty} w(z) f(z,t) \, \md z \leq e^{\bar{\alpha}\bar{k} t} \int_0^{\infty} w(z) f_I(z) \, \md z.
		$$
	\end{lemma}
	\begin{proof}
		From the weak formulation \eqref{BoE2weak} with the choice $\vp=w$ we immediately observe due to the monotonicity assumption on $w$ that
		\begin{align*}
			\frac{\md}{\md t} \int_0^{\infty} w(z) f(z,t) \, \md z = \int_0^{\infty} \int_0^z  f(z,t) f(y,t) k(z,y) \alpha(y,t)(w(z)-w(y))  \, \md y \, \md z \geq 0.
		\end{align*}
		From here we can further estimate 
		\begin{align*}
			\frac{\md}{\md t} \int_0^{\infty} w(z) f(z,t) \, \md z &\leq \bar{\alpha}\bar{k} \int_0^{\infty} \int_0^z  f(z,t) f(y,t)(w(z)-w(y))  \, \md y \, \md z \\
			&\leq \bar{\alpha}\bar{k} \int_0^{\infty} \int_0^z  f(z,t) f(y,t)w(z) \, \md y \, \md z,
		\end{align*}
		where we used the boundedness of $\alpha$ and $\kappa$ as well as the positivity of $w$. This provides the exponential bound on the first moment. Hence, blow-up in finite time cannot occur. 
	\end{proof}

\paragraph{Existence and uniqueness of a solution:}
	
	Due to mass conservation \eqref{masscons}, we can deal with the quadratic nonlinearity in the collision operator. This together with the boundedness of the collision kernel $k$ and the learning function $\alpha$ allows us to formulate the following existence result:
	
	\begin{theorem}\label{t:existenceBoE}
		Let the initial data $f_I$ fulfil \eqref{a:fI}, $\alpha \in L^{\infty}(\R_+^2)$ and $k(z,y) \leq \bar{k}$ for all $y<z$. Then \eqref{BoE2} has a unique global solution $f \in C\big([0,\infty),\,L_+^1(\R_+)\big)$.
	\end{theorem}
	
	\begin{proof}
		Let $f,g \in L_+^1(\R_+)$ such that $\|f\|_{L^1(\R_+)},\, \|g\|_{L^1(\R_+)} \leq 1$. We aim to perform Lipschitz estimates on the collision operator in order to be able to apply a Picard iteration argument. We treat the gain term $G$ and the loss-term $L$ defined in \eqref{BoE2} separately. We calculate
		\begin{align*}
			\|G(f,f)-G(g,g)\|_{L^1(\R)} &= \int_0^{\infty} \left|\int_0^z \alpha(z,t)k(z,y)\left(f(z,t)f(y,t)-g(z,t)g(y,t)\right)\, \md y \right| \, \md z \\
			&\leq \bar{\alpha} \bar{k} \int_0^{\infty}\int_0^z \left|f(z,t)f(y,t)-g(z,t)g(y,t)\right| \, \md y \, \md z,
		\end{align*}
		where we used that $\alpha(z,t) \leq \bar{\alpha} := \max_{(z,t) \in \R_+^2} \alpha(z,t)$ and the fact that $k(z,y) \leq \bar{k}$ for $y<z$. We further estimate
		\begin{align*}
			\|G(f,f)-G(g,g)\|_{L^1(\R)} &\leq \bar{\alpha} \bar{k} \int_0^{\infty}\int_0^{\infty} \left(|f(z,t)||f(y,t)-g(y,t)|+|g(y,t)||f(z,t)-g(z,t)|\right) \, \md y \, \md z \\
			&\leq 2\bar{\alpha}{\bar{k}} \|f-g\|_{L^1(\R_+)}.
		\end{align*}
		Similarly, we proceed with the loss-term
		\begin{align*}
			\|L(f,f)-L(g,g)\|_{L^1(\R)} &= \int_0^{\infty} \left|\int_z^{\infty} \alpha(y,t)k(y,z)\left(f(z,t)f(y,t)-g(z,t)g(y,t)\right)\, \md y \right| \, \md z \\
			&\leq \bar{\alpha}{\bar{k}} \int_0^{\infty}\int_z^{\infty} \left|f(z,t)f(y,t)-g(z,t)g(y,t)\right| \, \md y \, \md z,
		\end{align*}
		where we used again the boundedness of $\alpha$ and here the fact that $k(y,z) \leq \bar{k}$ for $z<y$. As before, we obtain
		\begin{align*}
			\|L(f,f)-L(g,g)\|_{L^1(\R)} &\leq \bar{\alpha} \int_0^{\infty}\int_0^{\infty} \left(|f(z,t)||f(y,t)-g(y,t)|+|g(y,t)||f(z,t)-g(z,t)|\right) \, \md y \, \md z \\
			&\leq 2\bar{\alpha} {\bar{k}}\|f-g\|_{L^1(\R_+)},
		\end{align*}
		from which we conclude
		\begin{align*}
			\|Q(f,f)-Q(g,g)\|_{L^1(\R+)} \leq 4 \bar{\alpha} \bar{k} \|f-g\|_{L^1(\R_+)},
		\end{align*}
		and, hence, Lipschitz-continuity of $Q$ uniform in time. From here, we can conclude the existence of a unique global solution by Picard iteration. Nonnegativity and conservation of mass follow immediately, the latter implying global existence by an iteration argument.
	\end{proof}
	
	\paragraph{Asymptotic behaviour:}
	
	We aim to show that under the assumption of limited initial knowledge, i.e. $\supp{(f_I)} = \mathcal{K}$, where $\mathcal{K} \subset \R_+$ is \emph{compact}, knowledge will concentrate at $z_*:=\arg \max_{z\in \R_+} f_I(z)$. 
	
	The first result shows that under this assumption the creation of more knowledge is impossible.
	\begin{lemma}\label{l:compsupp}
		Let $\supp{(f_I)} = \mathcal{K}$, where $\mathcal{K}$ is a compact subset of $\R_+$ and the solution to \eqref{BoE2} be continuous, i.e. $f \in C([0,\infty) \times \R_+)$. Then $\supp{(f(\cdot, t))} \subset \mathcal{K}$ for all $t>0$.
	\end{lemma}
	\begin{proof}
		This result is based on a maximum principle argument, similar to \cite{BLW}[Proposition 3.1].
	\end{proof}
	
	The increase of mean knowledge \eqref{mean} together with Lemma \ref{l:compsupp} implies that knowledge will accumulate at the largest initial knowledge level, that is $z_* = \sup \supp(f_I)$, or at $\infty$ if $f_I$ is positive everywhere. In particular, if we can ensure that learning is always possible, i.e. the kernel $k$ and the learning function $\alpha$ have a positive lower bound, the knowledge will accumulate at the highest level possible.
	
	\begin{theorem}\label{t:concentration}
		Let $z_* := \sup\big({\supp{(f)}}\big)$.
		Let further $\alpha(z,t) \geq \underline{\alpha} >0$, $k(z,y) \geq \underline{k} >0 $ for all $z<y< z_*$, $t>0$. If $z_*<\infty$, i.e. the support of $f_I$ is bounded, we have
		$$
		f(\cdot, t) \rightharpoonup^* \delta_{z_*}, \, \text{ for }\, t \to \infty.
		$$
		In the case where the $f(\cdot,t)$ does not have compact support the knowledge accumulates at $z=\infty$ for $t \to \infty$.
	\end{theorem}
	\begin{proof}
		For the cumulative distribution function of \eqref{F} we compute
		\begin{align*}
			- \frac{\md}{\md t} F(Z,t) &= \frac{d}{dt} (1-F(Z,t)) =  \frac{\md}{\md t} \int_Z^{\infty} f(z,t) \, \md z \\
			&= \int_Z^{\infty} \left[\int_0^z \alpha(y,t) k(z,y) f(y,t) f(z,t) \, \md y - \int_Z^z \alpha(y,t) k(z,y) f(y) f(z) \, \md y \right] \, \md z \\
			&=  \int_Z^{\infty} \int_0^Z \alpha(y,t) k(z,y) f(y) f(z) \, \md y \, \md z \\
			&\geq \underline{\alpha} \underline{k} \int_Z^{\infty} \int_0^Z f(y) f(z) \, \md y \, \md z = \underline{\alpha} \underline{k} (1-F(Z,t))F(Z,t),
		\end{align*}
		where from the first to the second line we used \eqref{BoE2weak} with $\vp(z) := \mathbb{1}_{[Z,\infty)}$ and the estimate is due to $\alpha(z,t) \geq \underline{\alpha}$ for all $z,t \in \R_+$ and $k(z,y) \geq \underline{k}$ for all $z > y$. This is equivalent to
		\begin{align*}
			\frac{\md}{\md t} F(Z,t) \leq - \underline{\alpha} \underline{k} (1-F(Z,t))F(Z,t),
		\end{align*}
		from which we can conclude uniform convergence $F(Z,t) \to 0$ for all $Z < z_*$ as $t \to \infty$. Convergence of the distribution function $f$ to the Dirac mass centric at $z_*$ follows.
	\end{proof}
	
	The condition of the initial datum having non-compact support is crucial to ensure that the overall knowledge level is increasing. In particular a Pareto tail condition on the initial datum is needed to allow exponential growth of the overall economy as in case of BGP solutions (see Section \ref{s:bgp}).
		
	\subsection{Investigation of the Hamilton-Jacobi-Bellman equation}\label{s:bellman}
	
	Next we investigate the Hamilton-Jacobi-Bellman equation
	\begin{equation}\label{BeE2}
		\begin{split}
			&\pa_t V(z,t) - rV(z,t) = -\max_{s \in \mathcal{S}}{\left[U((1-s)z) + \alpha(s(z,t)) \int_z^{\infty} \left(V(y,t)-V(z,t)\right)f(y,t)k(y,z) \, \md y \right]}, \\
			&V(z,T) = 0,
		\end{split}
	\end{equation}
	 for a given distribution function $f \in C\left((0,T], L_+^1(\R_+)\right)$. To ease notation, we define 
	\begin{align}\label{B}
		B(z,t) :=  \int_z^{\infty} \left(V(y,t)-V(z,t)\right)f(y,t)k(y,z) \, \md y, 
	\end{align}
	which measures the \emph{benefit from search} at time $t>0$ of an individual with knowledge level $z$.
	
	\paragraph{Solvability of the optimization problem:}
	First, we make sure that the maximization problem in \eqref{BeE2} has a solution for a fixed $B \in \R_+$.	
	\begin{lemma}\label{l:max}
		Let $\alpha:\, [0,1] \to \R_+$, satisfy Assumptions \ref{a:alpha} and $U:\, \R_+ \to \R_+$ satisfy Assumption \ref{a:U}.  Then for every $z>0$ and $B \in \R$ there exists a unique solution $S=S(B)$ to the optimization problem
		\begin{align}\label{OP}
			\max_{s \in \mathcal{S}} \left[U((1-s)z)+\alpha(s) B\right],
		\end{align}
		and we denote $S=S(B):=\operatorname{argmax}_{s \in \mathcal{S}} \left[U((1-s)z)+\alpha(s) B\right]$.
	\end{lemma}
	\begin{proof}
		Defining $\xi:=\alpha(s)$ and $\alpha^{-1}:=\beta$, we can rewrite \eqref{OP} as 
		\begin{align*}
			\max_{\xi \in [0,\alpha(1)]} \left[U((1-\beta(\xi)))z+B\xi\right].
		\end{align*}
		Calculating the optimality condition we obtain
		\begin{align}\label{optcond}
		\frac{B}{z} = U'((1-\beta(\xi))z) \beta'(\xi).
		\end{align}
		The assumption on $\alpha$ (strict concavity) and hence $\beta$ (strict convexity) imply
		$$
		\pa_{\xi} \left( U'((1-\beta(\xi))z) \beta'(\xi)\right) = -U''((1-\beta(\xi))z)\beta'^2(\xi)z+U'((1-\beta(\xi))z)\beta''(\xi)>0, \quad \forall \beta(\xi) \neq 1
		$$
		and, hence, existence of a unique solution $S \in [0,1)$ to \eqref{OP}.
	\end{proof}
	
	\begin{lemma}\label{l:lip}
		Let $\alpha:\, \R_+\to \R_+$, $U:\, \R_+ \to \R$ fulfil Assumptions \ref{a:alpha} and \ref{a:U}, respectively. Let further $S=S(B)$ be the optimal solution to \eqref{OP} for a given $z \in \R_+$ and $B \in \R$. Under the condition
		$$
		\lim_{B \to 0}B^3\frac{\left[\alpha''(S(B))U'((1-S(B))z)+\alpha'(S(B))zU''((1-S(B))z)\right]}{U'((1-S(B))z)^3} < 0,
		$$
		the maps $B \to S(B)$, $B \to \alpha(S(B))$ and $\alpha(S(B))B$ are Lipschitz-continuous.
	\end{lemma}
	\begin{proof}
		We distinguish between the following cases
		\begin{enumerate}
			\item{$B\leq0$:} For non-positive $B$ we see immediately that $S(B) \equiv 0$ has to hold.
			\item{$0<B<\frac{zU'(0)}{\alpha'(1)}$:} In this case there exists a unique solution $S(B)$ given by $S(B)=H^{-1}(\frac{z}{B})$, where we defined $H(S(B))):=\frac{\alpha'(S(B))}{U'((1-S(B))z)}$. The invertibility of the function $H(\cdot)=\frac{\alpha'(\cdot)}{U'((1-\cdot)z)}$ is secured since it is strictly monotonically increasing due to the monotonicity assumptions we made on $\alpha'$ and $U'$. Derivation of $S(B)$ yields
			$$
			S'(B)= -\frac{zU'((1-S(B))z)^2}{B^2\alpha''(S(B))U'((1-S(B))z)+B^2zU''((1-S(B))z)\alpha'(S(B))}.
			$$
			Furthermore, we compute
			\begin{align*}
				&\frac{\md}{\md B} \alpha(S(B))= \alpha'(S(B))S'(B)=\frac{zU'((1-S(B))z)}{B}S'(B)\\
				 &\quad =-\frac{z^2U'((1-S(B))z)^3}{B^3\left[\alpha''(S(B))U'((1-S(B))z)+\alpha'(S(B))zU''((1-S(B))z)\right]},
			\end{align*}
			from which together with the assumption 
			$$
			\lim_{B \to 0}\frac{U'((1-S(B))z)^3}{B^3\left[\alpha''(S(B))U'((1-S(B))z)+\alpha'(S(B))zU''((1-S(B))z)\right]} < 0
			$$
			we obtain positivity of $S'(B)$ for small $B$. This gives piece-wise continuous differentiability of $S$ as a function of $B$. Since due to the aforesaid assumption also  
			$$
			-\lim_{B \to 0}B^2\frac{\left[\alpha''(S(B))U'((1-S(B))z)+\alpha'(S(B))zU''((1-S(B))z)\right]}{U'((1-S(B))z)^3} = \infty
			$$
			has to hold, we conclude (using Assumption \ref{a:U} for $U'$) that
			$$
			\lim_{B\to0} S'(B) = 0.
			$$
			Therefore $S$ is continuous at $0$. Continuity at $\frac{zU'(0)}{\alpha'(1)}$ for finite $U'(0)$ follows from the continuity of $\alpha', \alpha'', U', U''$ on $(0,\infty)$. Thus, we can deduce Lipschitz-continuity for $B \to S(B)$, which moreover ensures Lipschitz-continuity of $B \to \alpha(S(B))$ and $B \to B \alpha(S(B))$. We want to point out that the condition $U'(0)$ finite is not a restriction, since if $U'(0)=\infty$ the control $S$ can never reach the value 1, which can be seen from the optimality condition \eqref{optcond}.
			\item{$B > \frac{zU'(0)}{\alpha'(1)}$:} This case is only relevant if $\lim_{y \to 0}U'(y)<\infty$, as it is the case for the linear utility function \eqref{e:ulin}. Taylor-expansion of the concave functions $\alpha$ around 1 and $U$ around 0 in the subjective function of \eqref{OP} gives 
			\begin{align*}
				U(z(1-s))+B\alpha(s) &\leq U(0) + U'(0)z(1-s)+B\alpha(1)+B\alpha'(1)(s-1) \\
				&=U(0)B\alpha(1) + (1-s)(U'(0)z-B\alpha'(1)) <U(0)+B\alpha(1),
			\end{align*}
			which is equivalent to 
			\begin{align*}
				U(z(1-s))+B\alpha(s) < U(0)+B\alpha(1), \quad \forall s \in [0,1].
			\end{align*}
			Hence, the maximum is attained at $s=1$, form which we conclude $S(B)\equiv 1$ for $B>\frac{zU'(0)}{\alpha'(1)}$.
		\end{enumerate}
	\end{proof}
	
	\paragraph{Existence and uniqueness of a solution:}
	We conclude the investigation of the value function $V$ with an existence result of \eqref{BeE2} for given population density $f$. Indeed, due to the structure of the benefit from search and the Lipschitz properties (Lemma \ref{l:lip}) we are able to prove the existence result for the Hamilton-Jacobi-Bellman equation:
	
	\begin{theorem}\label{t:existenceBeE}
		Let $f \in C((0,T);L^1(\R_+))$ be given and assume the search function $\alpha$ satisfies Assumption \ref{a:alpha} and the utility function $U$ satisfies Assumption \ref{a:U} with $U \in C\big([0,\infty)\big)$. Then there exists a unique solution $V \in C((0,T);L^{\infty}(\R_+))$ of \eqref{BeE2}. Moreover, if $\tilde{V}$ is a solution to \eqref{BeE2} with $\tilde{f}$, then there exist constants $m$ and $D$ (independent of $\tilde{V}$ and $\tilde{f}$), such that 
		$$	
		\|V-\tilde{V}\|_{L^\infty} \leq De^{mt}\|f-\tilde{f}\|_{C((0,T);L^1(\R_+))}\|\tilde{V}\|_{L^{\infty}}.
		$$
	\end{theorem}
	\begin{proof}
		The proof can be performed in similar manner as in \cite{BLW}{[Theorem 3.9]}. One can see easily that the mapping $V \to B(V)$ is Lipschitz continuous in $L^{\infty}(\R_+)$ due to the boundedness of the learning kernel and $f \in L^1(\R_+)$. Moreover, thanks to the properties proved in Lemma \ref{l:lip} we see that the mappings $B(V) \to S$ and $B(V) \to \alpha(S)B$ inherit these Lipschitz-continuities. Hence, the right-hand-side of \eqref{BeE2} is Lipschitz and the claim follows from a Picard argument. The contraction property can be seen by careful estimation of $\pa_t(V-\tilde{V})$.
	\end{proof}
	
	\begin{rem}
		We would like to point out that the above existence result is not valid for the problem with logarithmic utility due to the lack of control at $z=0$, which can be seen in Section \ref{s:numerics}, Figure \ref{f:log_v}. 
	\end{rem}
	
	\paragraph{Qualitative properties of $V$:}
	With the existence of an \emph{optimal time allocation} for every fixed \emph{benefit from search} $B$ we further aim to investigate the qualitative behaviour $V$ and $S$. For ease of notation recall the definition
	\begin{align}\label{S}
		S(z,t):={\arg \max}_{s \in \mathcal{S}}{\left[U((1-s(z,t))z) + \alpha(s(z,t)) \int_z^{\infty} \left(V(y,t)-V(z,t)\right)f(y,t)k(y,z) \, \md y \right]},
	\end{align}
	and write for \eqref{BeE2} at its minimum 
	\begin{align}\label{BeES}
		\pa_t V(z,t) - rV(z,t) = -U((1-S(z,t))z) + \alpha(S(z,t)) \int_z^{\infty} \left(V(y,t)-V(z,t)\right)f(y,t)k(y,z) \, \md y.
	\end{align}
		This allows us to rewrite \eqref{BeES} as
	\begin{align*}
		\pa_t V(z,t)-rV(z,t) = -U((1-S(z,t))z) - \alpha(S(z,t))B(z,t). 
	\end{align*}
	Moreover, we compute the derivative w.r.t. $z$ of $B$ \eqref{B}, which has the following form
	\begin{align}\label{dB}
		\pa_zB(z,t) = - \pa_zV(z,t) \int_z^{\infty} f(y,t)k(y,z) \, \md y + \int_z^{\infty} \left(V(y,t)-V(z,t)\right)f(y,t)\pa_zk(y,z) \, \md y.
	\end{align}

	\begin{lemma}\label{l:Vpos}
		Let $U(p)\geq0$ for all $p \in [0,\infty)$, then a solution $V(\cdot,t)$ of \eqref{BeE2} is non-negative. 
	\end{lemma}
	
	\begin{proof}
		Let $z_0$ be a minimal point of $V(\cdot,t)$, i.e $V(z,t) \geq V(z_0,t)$ for all $z \in \R_+$. Then we have from \eqref{BeES} 
		\begin{align*}
			\pa_t V(z_0,t) - rV(z_0,t) &\leq -U(1-S(z_0,t))z_0 - \alpha(S(z_0,t)) \int_{z_0}^{\infty} \left(V(z_0,t)-V(z_0,t)\right)f(y,t)k(y,z_0) \, \md y \\
			&=-U(1-S(z_0,t))z_0,
		\end{align*}
		which can be written as
		\begin{align*}
			rV(z_0,t) \geq U(1-S(z_0,t))z_0 + \pa_tV(z_0,t).
		\end{align*}
		Since the above differential inequality evolves backward in time, nonnegativity of $V(z_0,t)$ hence of $V(z,t)$, for all $z \in \R_+$ is preserved. 
	\end{proof}
	
	\begin{rem}
		Since the logarithm does not fulfill the crucial positivity assumption, the above Lemma fails in the case of the logarithmic utility function \eqref{e:ulog}. Indeed, the value function $V$ cannot be expected to be positive for small $z$, which will also be confirmed by the numerical simulations in Section \ref{s:numerics}. We will see that a condition for the knowledge level $z_0$ at which the value function $V$ equals zero (and changes sign) can be given (see Lemma \ref{l:Vposln} below).
	\end{rem} 
	To show \emph{monotonicity of the value function} $V(\cdot,t), t\geq0$ we need the following auxiliary observation:
	\begin{lemma}\label{l:hilfslemma}
		Let $h=h(y,z,t) >0$ for all $z,y,t \in \R_+$ and let $u=u(z,t)$ fulfill
		\begin{align*}
			&\pa_tu(z,t) \leq - \int_z^{\infty} u(y,t) h(y,z,t) \, \md y, \quad \text{ for all } z,t \geq 0 \\
			&u(z,T) \geq 0, \quad \text{ for all } z\geq 0, \text{ for a time } T \geq 0.		
		\end{align*}
		Then $u(z,t)\geq 0$ for all $t\in [0,T]$.
	\end{lemma}
	\begin{proof}
		We write $u=v^{\ve}+\ve \vp$, with $\ve >0$ and $\vp$ fulfilling 
		\begin{align*}
			&\pa_t \vp(z,t) = 1 - \int_z^{\infty} \vp(y,t) h(y,z,t) \, \md y \\
			&\vp(\cdot,T)=0.
		\end{align*}
		For $v^{\ve}$ it holds $v^{\ve}(\cdot,T) \geq 0$ and further
		\begin{align*}
			\pa_tv^{\ve}(z,t) &= \pa_t u(z,t) - \ve \pa_t \vp(z,t) \\
			&\leq - \int_z^{\infty} v^{\ve}(y,t)h(y,z,t) \, \md y - \ve.
		\end{align*}
		Let us now assume that $v^{\ve}$ can become negative. Hence, we assume that there exists a $t_0 \in [0,T]$ such that $v^{\ve}(z,t_0) \geq 0$ for all $z \geq 0$ and a $z_0$ with $\pa_tv^{\ve}(z_0,t_0) \geq 0$. Then the above inequality yields
		$$
		0 \leq \pa_t v^{\ve}(z_0,t_0) \leq -\ve,
		$$
		which is a contradiction to $\ve >0$. Hence, the time-derivative of $v^{\ve}$ can never become non-negative, which together with the terminal condition implies the desired non-negativity of $v^{\ve}$. Non-negativity of $u(z,t)$ for all $z \geq 0$, $t \in [0,T]$ follows since the above arguments hold independently of $\ve >0$.
	\end{proof}
	\begin{theorem}\label{t:Vmon}
		A solution $V(\cdot,t)$ of \eqref{BeE2} is non-decreasing for all times $t \in [0,T)$. 
	\end{theorem}
	\begin{proof}
		Differentiating the Hamilton-Jacobi-Bellman equation evaluated at its minimum $S$ \eqref{BeES} with respect to $z$, we obtain the following equation for $W(z,t):=\pa_zV(z,t)$:
		\begin{align*}
			\pa_t W(z,t) - rW(z,t) 
			&= -(1-S(z,t))U'((1-S(z,t))z) + \alpha(S(z,t)) \int_z^{\infty} W(z,t)f(y,t) k(y,z)\, \md y \\
			&\phantom{=}- \alpha(S(z,t)) \int_z^{\infty} \left(V(y,t)-V(z,t)\right)f(y,t) \pa_zk(y,z)\, \md y,
		\end{align*}
		where we used the optimality condition of $S(z,t)$
		$$
		\alpha'(S(z,t)) = z U'((1-S(z,t)z))\left( \int_z^{\infty} \left(V(y,t)-V(z,t)\right)f(y,t)k(y,z) \, dy\right)^{-1}.
		$$
		Since
		\begin{align*}
			\int_z^{\infty} &\left(V(y,t)-V(z,t)\right)f(y,t) \pa_zk(y,z)\, dy = \int_z^{\infty} W(x,t)  \int_{x}^{\infty}  f(y,t) \pa_zk(y,z)\, \md y \,v\md x
		\end{align*}
		and defining
		\begin{align*}
			A(z,t):= r +  \alpha(S(z,t)) \int_{x}^{\infty} f(y,t) k(y,z)\, \md y \geq 0,
		\end{align*}
		we have
		\begin{align*}
			\pa_t W(z,t) - A W(z,t) = &-(1-S(z,t))U'((1-S(z,t))z) \\
			&- \alpha(S(z,t))\int_z^{\infty} W(x,t)  \int_{x}^{\infty}  f(y,t) \pa_zk(y,z)\, \md y \, \md x.
		\end{align*}
		Setting $u(z,t) = \exp{\left(\int_t^TA(z,\tau) \, \md \tau \right)} W(z,t)$, we obtain
		\begin{align*}
			&\exp{\left(-\int_t^TA(z,\tau) \, \md \tau \right)} \pa_t u(z,t) = -A w(z,t) + \pa_tW(z,t) \\
			&\phantom{exp}=-(1-S(z,t))U'((1-S(z,t))z) - \alpha(S(z,t)) \int_z^{\infty}W(x,t) \int_{x}^{\infty}  f(y,t) \pa_zk(y,z)\, \md y \, \md x\\
			&\phantom{exp}\leq - \alpha(S(z,t)) \int_z^{\infty}W(x,t) \int_{x}^{\infty}  f(y,t) \pa_zk(y,z)\, \md y \, \md x \\
			&\phantom{exp}\leq - \alpha(S(z,t)) \int_z^{\infty}\exp{\left(-\int_t^TA(x,\tau) \, \md \tau \right)}  u(x,t) \int_{x}^{\infty}  f(y,t) \pa_zk(y,z)\, \md y \, \md x,
		\end{align*}
		which is equivalent to
		\begin{align*}
			\pa_t u(z,t)\leq - \alpha(S(z,t)) \int_z^{\infty}\exp{\left(\int_t^T(A(z,\tau)-A(x,\tau)) \, \md \tau \right)}  u(x,t) \int_{x}^{\infty}  f(y,t) \pa_zk(y,z)\, \md y \, \md x.
		\end{align*}
		By defining 
		$$
		h(z,x,t):=\exp{\left(\int_t^T(A(z,\tau)-A(x,\tau)) \, \md \tau \right)} \int_{x}^{\infty}  f(y,t) \pa_zk(y,z)\, \md y
		$$
		we can use Lemma \ref{l:hilfslemma} to conclude non-negativity of $u(z,t)$ and, hence, of $W(z,t):=\pa_zV(z,t)$.
	\end{proof}
	
	With this monotonicity result for $V(\cdot,t)$ we are also able to state the condition where $V(\cdot,t)$ changes its sign for logarithmic utility \eqref{e:ulog}.
	
	\begin{lemma}\label{l:Vposln}
		For equation \eqref{BeE2} with $U(p)=\ln{(p)}$ let us further assume that there exists a $z_0 \in \R_+$, s.t. 
		$$
		\ln\left((1-S(z_0,t))z_0\right) = -\alpha(S(z_0,t)) B(z_0,t).
		$$
		Then the solution $V(\cdot,t)$ of \eqref{BeE2} is non-negative for all $z > z_0$ and $V(z_0,t)=0$, for all $t\in[0,T]$. 
	\end{lemma}
	
	\begin{proof}
		By definition we have that equation \eqref{BeE2} at $z=z_0$ reduces to
		\begin{align*}
			\pa_tV(z_0,t) = rV(z_0,t),
		\end{align*}
		and therefore  $V(z_0,t)=0$ follows from the terminal condition $V(\cdot,T)=0$. Monotonicity of $V(\cdot,t)$ (Lemma \ref{t:Vmon}) then implies that $V(z,t) \geq 0$, for all $z > z_0$.
	\end{proof}	

	\paragraph{Qualitative properties of $S$:}
	
	We conclude Section \ref{s:bellman} with investigation the monotonicity of the control $S(\cdot,t)$. We start with pointing out that due to the optimality condition
	    $$
	    	S(z,t)=H^{-1}\left(\frac{z^{1-\zeta}}{B(z,t)}\right), \quad H^{-1}(\cdot) \searrow \,,
	    $$
	    the monotonicity of $S(\cdot, t)$ strongly depends on the monotonicity of $B(\cdot, t)$. Moreover, investigating \eqref{dB} leaves the strong indication that the monotonicity of $B(\cdot,t)$ strongly depends on the monotonicity of $\pa_zV(\cdot,t)$, hence the sign of $\pa_z^2V(\cdot,t)$. However, although we fail to determine either of these with the techniques and conditions presented in the framework of this article, we are able to prove monotonicity of $S(\cdot, t)$ by proving monotonicity of $B(\cdot, t)$ multiplied by an integrating factor and restricting ourselves to certain parameter regimes depending on the learning kernel and utility function. 
	
	\begin{lemma}\label{l:BpBe}
		Let Assumption \ref{a:alpha} on the search function $\alpha$ be satisfied. Then the benefit from search $B(\cdot,t)$, defined in \eqref{B}, is non-negative. Moreover, 
		\begin{enumerate}[label={(\roman*)},itemindent=1em]
		\item in the case of the polynomial learning kernel \eqref{e:kpoly} the function
		\begin{align}		
			\label{e:polytildeB}
			\tilde{B}_p(z,t):=z^{-\kappa(1-\delta)}B(z,t)\, \quad z \in [0,\infty),
		\end{align}
		is non-increasing, \label{l:poly}
		\item and for the exponential learning kernel \eqref{e:kexp} the function 
		\begin{align}\label{e:exptildeB}
			\tilde{B}_e(z,t):= e^{-\kappa z}B(z,t), \quad z \in [0,\infty),
		\end{align} 
		is non-increasing. \label{l:exp}
		\end{enumerate}
	\end{lemma}
	\begin{proof}
		From its definition \eqref{B} and the monotonicity of $V(\cdot,t)$, cf. Lemma \ref{t:Vmon}, we can see immediately that $B(\cdot,t) \geq 0$. 
		
		For the proof of \ref{l:poly} we notice that since 
		$$
		\pa_zk(y,z) = \frac{\kappa (1-\delta)}{z} \left(k(y,z)-\delta \right) 
		$$
		we can write \eqref{dB} as 
		\begin{align*}
			\pa_zB(z,t) &= - \pa_z V(z,t) \int_z^{\infty} f(y,t) k(y,z) \, \md y - \frac{\delta (1-\delta)\kappa}{z} \int_z^{\infty} \left(V(y,t)-V(z,t)\right) f(y,t) \, \md y \\
			&+ \frac{\kappa (1-\delta)}{z} B(z,t), 
		\end{align*}
		which is equivalent to 
		\begin{align}\label{BODE}
			\pa_zB(z,t) -  \frac{\kappa (1-\delta)}{z} B(z,t) = - g_p(z,t),
		\end{align}
		where we defined 
		$$
		g_p(z,t):= \pa_z V(z,t) \int_z^{\infty} f(y,t) k(y,z) \, \md y + \frac{\delta (1-\delta)\kappa}{z} \int_z^{\infty} \left(V(y,t)-V(z,t)\right) f(y,t) \, \md y\,
		$$
		which is non-negative, since $V(\cdot,t)$ is non-increasing for all $t \in [0,T)$. Multiplying $\eqref{BODE}$ by the integrating factor $z^{-\kappa(1-\delta)}$, we obtain
		\begin{align*}
			\pa_zB(z,t)z^{-\kappa(1-\delta)} -  \kappa (1-\delta) z^{-\kappa(1-\delta)-1}B(z,t) = \pa_z(z^{-\kappa(1-\delta)}B(z,t)) = -z^{-\kappa(1-\delta)} g_p(z,t) \leq 0,
		\end{align*}
		and therefore $\tilde{B}_p(z,t)$ is non-increasing. 
		
		In order to prove \ref{l:exp} we notice in a similar manner that since 
		$$
		\pa_zk(y,z) =  \kappa k(y,z),
		$$
		we can write \eqref{dB} as 
		\begin{align*}
			\pa_zB(z,t) = - \pa_z V(z,t) \int_z^{\infty} f(y,t) k(y,z) \, \md y + \kappa B(z,t), 
		\end{align*}
		which is equivalent to 
		\begin{align}\label{BODEexp}
			\pa_zB(z,t) -  \kappa B(z,t) = - g_e(z,t),
		\end{align}
		where we defined 
		$$
		g_e(z,t):= \pa_z V(z,t) \int_z^{\infty} f(y,t) k(y,z) \, \md y,
		$$
		which is non-negative, since $V(\cdot,t)$ is non-increasing for all $t \in [0,T)$. Multiplying $\eqref{BODEexp}$ by the integrating factor $e^{-\kappa z}$, we obtain
		\begin{align*}
			\pa_z(e^{-\kappa z}B(z,t)) = -e^{-\kappa z} g_e(z,t) \leq 0,
		\end{align*}
		and hence, $\tilde{B}(z,t):=e^{-\kappa z}B(z,t)$ is non-increasing. 
	\end{proof}
	
	With this preliminary result we are able to conclude monotonicity of the learning function $S$ depending on the choice of the utility function $U$. For our considerations we restrict ourselves to the \emph{isoelastic utility} functions \eqref{e:ucrra} for $\zeta \in [0,1)$, including the limiting linear case $\zeta=0$ defining the linear utility.
	
	\begin{cor}\label{c:Smon}
	Let all assumptions of Lemma \ref{l:BpBe} hold. 
	\begin{enumerate}[label={(\roman*)},itemindent=1em]
		\item Furthermore, in the case of the polynomial learning kernel \eqref{e:kpoly}, let the constants $\kappa, \delta$ as well as $\zeta \in (0,1)$ for the the isoelastic utility \eqref{e:ucrra} and $\zeta =0$ for the linear utility \eqref{e:ulin} satisfy
		\begin{align}\label{kappa_delta}
			1-\zeta \geq \kappa(1-\delta).
		\end{align}
		Then the maximizer $S(\cdot,t)$ is non-increasing for all $t \in [0,T)$. Moreover, $S(\cdot,t)$ is strictly decreasing, whenever $0<S(\cdot,t)<1$. \label{c:Smonp}
		\item In the case of the exponential learning kernel \eqref{e:kexp}, the maximizer $S(z,t)$ with isoelastic utility \eqref{e:ucrra}, including the limiting case $\zeta=0$, is non-increasing for all $t \in [0,T)$ and $z \in \left[0,\frac{1-\zeta}{\kappa}\right]$. \label{c:Smone}
	\end{enumerate}		
	\end{cor}
	
	\begin{proof}
		As in Lemma \ref{l:lip} we distinguish between the three cases:
		\begin{itemize}
			\item If $B(z,t)>\frac{zU'(0)}{\alpha'(1)}$, only relevant if $U'(0)<\infty$, we showed in Lemma \ref{l:lip} that $S(z,t)=1$. 
			\item For $0<B(z,t)<\frac{zU'(0)}{\alpha'(1)}$, we have 
			$$
				S(z,t)=H^{-1}\left(\frac{z^{1-\zeta}}{B(z,t)}\right), 
			$$
			where, as in the proof of Lemma \ref{l:lip}, we used the notation
			$$	
				H(S)=\alpha'(S)(1-S)^{\zeta},
			$$
			which is strictly decreasing due to Assumptions \ref{a:alpha} and \ref{a:U}. 
			For proving case \ref{c:Smonp} we notice
			$$
				S(z,t)=H^{-1}\left(\frac{z^{1-\zeta}}{B(z,t)}\right)=H^{-1}\left(\frac{z^{1-\zeta-\kappa(1-\delta)}}{z^{-\kappa(1-\delta)}B(z,t)}\right) = H^{-1}\left(\frac{z^{1-\zeta-\kappa(1-\delta)}}{\tilde{B}_p(z,t)}\right).
			$$
			Hence, we can conclude the desired strict monotonicity of $S(\cdot,t)$, since $H^{-1}$ is strictly decreasing and its argument is non-decreasing under assumptions \eqref{kappa_delta}. In the case of the exponential learning kernel \ref{c:Smone} we can write
			$$
				S(z,t)=H^{-1}\left(\frac{z^{1-\zeta}}{B(z,t)}\right)=H^{-1}\left(\frac{z^{1-\zeta}e^{-\kappa z}}{e^{-\kappa z}B(z,t)}\right)=H^{-1}\left(\frac{z^{1-\zeta}e^{-\kappa z}}{\tilde{B}_e(z,t)}\right).
			$$
			Again using the monotonicity of $H^{-1}$, we can conclude that $S(z,t)$ is monotonically decreasing in $z$, if $z \leq \frac{1-\zeta}{\kappa}$.
			\item If we start with $B(0,t)=0$ then $\tilde{B}_p(0,t)=\tilde{B}_e(0,t)=0$ and therefore $\tilde{B}_p(z,t), \tilde{B}_e(0,t) \leq 0$ for all $z > 0$. This implies that $B(z,t) \leq 0$. Then it is clear from \eqref{OP} that $S(z,t) = 0$ has to hold for all $z,t>0$.
		\end{itemize}
	\end{proof}
	
	\begin{rem}
		\begin{itemize}
		\item We want to point out that in the case of the polynomial learning kernel \eqref{e:kpoly} the restrictions on the constants \eqref{kappa_delta} are sufficient for the monotonicity of $S(\cdot, t)$, but not necessary. We will see in Section \ref{s:numerics} that the bounds are not sharp and that we observe monotonicity of solutions for parameter sets violating condition \eqref{kappa_delta}.
		\item Since also in the case of the exponential learning kernel \eqref{e:kexp} we just worked with knowing the monotonicity of the integrating factor $e^{\kappa z}B(z,t)$, we assume that the conditions for monotonicity of the control $S$, stated in in Corollary \ref{c:Smon}\ref{c:Smone} are very likely too strong.  
		\end{itemize}
	\end{rem}

	\begin{rem}\label{r:monB}
		Again the limiting logarithmic case (i.e. $\zeta=1$) is not covered by the above result. In the regime $0<B(z,t)<\frac{zU'(0)}{\alpha'(1)}$ the optimality condition of \eqref{OP} reads as
		$$
		\alpha'(S(z,t))(1-S(z,t))=\frac{1}{B(z,t)}.
		$$ 
		Since both $\alpha'(S)$ and $1-S$ are monotonically decreasing, the function $S(\cdot,t)$  inherits the monotonicity of $B(\cdot,t)$ completely. 
	\end{rem}
	
	\subsection{Analysis of the coupled system}\label{s:analysis_coupled_systems}
	
	Similar as in \cite{BLW} we use a fixed-point argument to prove existence and uniqueness of a solution for the fully coupled Boltzmann mean-field game system \eqref{bmfg} in a specific parameter regime, compatible with the results of Section \ref{s:bellman}. The utility term in \eqref{bmfgV} suggests that $V$ grows w.r.t. $z$ at a rate that is proportional to $U$. Hence, we expect that $\frac{V(z,t)}{1+U(z)}$ is in $C((0,T),L^{\infty}(\R_+))$, where again we have to exclude the case of the logarithmic utility. We introduce the weighted spaces
	\begin{align*}
	    L^p_w(\R^+) = \left\lbrace u:\, \R_+ \rightarrow \R_+ \colon \, \frac{u}{w} \in L^p(\R_{+})  \right\rbrace,
	\end{align*}
	and define $\mathscr{L}^{\infty}(\R^+)$ for the weight $w_1 = 1+U(\cdot)$ and $p=\infty$ and $\mathscr{L}^{1}(\R^+)$ for $w_2 = \frac{1}{1+U(\cdot)}$ and $p=1$.
	\begin{rem}
	Lemma \ref{l:weightbound} ensures that $f(\cdot,t) \in \mathscr{L}^{1}(\R_+)$, where $f$ is a solution to \eqref{BoE2} with initial data satisfying $\int_0^{\infty} (1+U(z)) f_I(z) \, \md z < \infty$ and non-negative utility fulfilling Assumption \ref{a:U}.
	\end{rem}
	For the existence proof, we further need a bound on the solution $V$ of the Hamilton-Jacobi-Bellman equation \eqref{BeE2} stated in the following Lemma.
	\begin{lemma}\label{l:boundLinfty}
		Let $V(z,t)$ be the solution to \eqref{bmfgV} with non-negative utility function $U \in C\left([0,\infty)\right)$, which further fulfils Assumption \ref{a:U}. Then $V(\cdot, t)$ is bounded in $\mathscr{L}^{\infty}(\R^+)$ for all $t \in [0,T)$.
	\end{lemma} 
	\begin{proof}
		We will write $B(V,f)$ to emphasize the dependence of the benefit function $B$, defined in \eqref{B}, on $V$ and $f$. Then
		\begin{align*}
			B(V,f) &\leq \bar{k} \int_z^{\infty} V(y,t) f(y,t) \, \md y \leq \bar{k} \|V(\cdot,t)\|_{\mathscr{L}^{\infty}}\int_z^{\infty} (1+U(y)) f(y,t) \, \md y \\
			&\leq \bar{k} \|V(\cdot,t)\|_{\mathscr{L}^{\infty}} \|f(\cdot,t)\|_{\mathscr{L}^{1}}.
		\end{align*}
		The coordinate transform $t \mapsto -t$ in \eqref{BeE2} together with the non-negativity of $V$ as well as the monotonicity of $U$ and $S \in [0,1]$ yield
		$$
			\pa_t V(z,t) \leq U(z) + \alpha(S(z,t)) B(V,f)(z,t).
		$$
		After division of the inequality above by $1+U(z)$ and using the estimate on the benefit from search function we obtain
		\begin{align*}
			\pa_t \frac{V(z,t)}{1+U(z)} &\leq \frac{U(z)}{1+U(z)} + \frac{\bar{\alpha}\bar{k}}{1+U(z)}  \|f(\cdot,t)\|_{\mathscr{L}^1} \|V(\cdot,t)\|_{\mathscr{L}^{\infty}}  \\
			&\leq 1 + \bar{\alpha}\bar{k}\left(1+e^{\bar{\alpha}\bar{k}t} \|f_I\|_{\mathscr{L}^1}\right) \|V(\cdot,t)\|_{\mathscr{L}^{\infty}},
		\end{align*}
		where we used Lemma \ref{l:weightbound} in the last step.
	\end{proof}
	We are now able to prove the main theorem of this section:

	\begin{theorem}\label{t:existence}
		Let $f_I$ fulfill Assumption \ref{a:fI} as well as 
		$$
			\int_0^{\infty} U(z) f_I(z) \, \md z < \infty.
		$$
		Moreover let the search function $\alpha$ satisfy Assumption \ref{a:alpha} and the utility function $U$ be non-negative and fulfill Assumption \ref{a:U}. Then there exists a unique solution $(f,V) \in C\left((0,T);\mathscr{L}^{1}(\R^+) \times \mathscr{L}^{\infty}(\R^+)\right)$ of the fully coupled system \eqref{bmfg} for final time $T$ small enough.
	\end{theorem}
	\begin{proof}
		We define the mapping 
		\begin{align*}
		\mathcal{F}:\, C\left([0,T);\mathscr{L}^{1}(\R^+)\right) &\to C\left([0,T);\mathscr{L}^{1}(\R^+)\right), \\ 
		g &\overset{\mathcal{F}}{\rightarrow} f
		\end{align*}
		given by first solving 
		\begin{align}\label{F1}
			\pa_t V(z,t)  = r V(z,t) - \max_{s \in \mathcal{S}}{\left[U((1-s)z) + \alpha(s(z,t)) B(V,g)(z,t)\right]},
		\end{align}
		for $V$ and determining the maximizer $S \in \mathcal{S}$. For this optimal time allocation $S$ we solve 
		\begin{align}\label{F2}
			\pa_t f(z,t) &= f(z,t) \int_0^z \alpha(S(y,t)) k(z,y)f(y,t) \, \md y-\alpha(S(z,t)) f(z,t) \int_z^{\infty} k(y,z)f(y,t) \, \md y 
		\end{align}
		for $f$. That $\mathcal{F}$ is a self-mapping is obvious due to Lemma \ref{l:weightbound}, Theorem \ref{t:existenceBoE} and Theorem \ref{t:existenceBeE}. What remains to show is that $\mathcal{F}$ is a contraction. Taking the difference on the right-hand-side of \eqref{F1} for data $g, \tilde{g} \in C\left([0,T);\mathscr{L}^{1}(\R^+)\right)$ we obtain at their maxima $S(B), S(\tilde{B})$
		\begin{align}\label{Vlip}
			&\pa_t(V-\tilde{V}) = r(V-\tilde{V})  + \bigg(U((1-S(\tilde{B}))z)-U((1-S(B))z)\bigg) + \bigg(\alpha(S(\tilde{B}))B(\tilde{V},\tilde{g})-\alpha(S(B))B(V,g)\bigg) \notag \\ 
			&\qquad  \leq r(V-\tilde{V})  + \bigg(zU'((1-S(B))z)(S(\tilde{B})-S(B))\bigg)+ \bigg(\alpha(S(\tilde{B}))B(\tilde{V},\tilde{g})-\alpha(S(B))B(V,g)\bigg) \notag \\
			& \qquad \leq r(V-\tilde{V})  + \bigg(L_1zU'((1-S(B))z) + L_2\bigg)\left(\tilde{B}-B\right). 
		\end{align}
		where we used the concavity of $U$ from Assumption \ref{a:U} as well as the Lipschitz continuity of $B \to S(B)$ and $B \to \alpha(S(B))B$ from Lemma \ref{l:lip}. 
		The difference $|\tilde{B}-B|$ can be controlled in the following way
		\begin{align*}
			|\tilde{B}-B| &\leq \left|\int_z^{\infty} k(y,z)\left(V(y,t)g(y,t)-\tilde{V}(y,t)\tilde{g}(y,t)\right)\, \md y\right| \\
			&\phantom{\leq}+ \left|V(z,t)\int_z^{\infty} k(y,z)g(y,t) \, \md y - \tilde{V}(z,t)\int_z^{\infty}k(y,z)\tilde{g}(y,t)\, \md y\right| \\
			&\leq 2\bar{k} \left(\|V\|_{\mathscr{L}^{\infty}}\|g-\tilde{g}\|_{\mathscr{L}^{1}} + \|V-\tilde{V}\|_{\mathscr{L}^{\infty}}\|\tilde{g}\|_{\mathscr{L}^{1}}\right).
		\end{align*}
		Division of \eqref{Vlip} by $1+U(z)$ and taking the $\sup_{z \in R_+}$ we obtain the following estimate
		\begin{align*}
			\pa_t\|V-\tilde{V}\|_{{\mathscr{L}^{\infty}}} &\leq \left(r+\|\tilde{g}\|_{\mathscr{L}^{1}}\right)\|V-\tilde{V}\|_{\mathscr{L}^{\infty}} \\
			&\phantom{\leq}+ 2\bar{k}\|V\|_{\mathscr{L}^{\infty}}\sup_{z\in \R_+} \frac{L_1zU'((1-S(B))z) + L_2}{1+U(z)}\|g-\tilde{g}\|_{\mathscr{L}^{1}}.
		\end{align*}
		We notice that \eqref{a:U} ensures that $zU'((1-S(B))z) < \infty$ holds. 
		Hence, if 
		$
			\int_0^T \|g-\tilde{g}\|_{\mathscr{L}^{1}} \, dt 
		$
		is sufficiently small, the same holds true for $\|V-\tilde{V}\|_{\mathscr{L}^{\infty}}$. For \eqref{F2} we estimate
		\begin{align*}
			&\frac{\md}{\md t} \|f-\tilde{f}\|_{\mathscr{L}^{1}} \\
			&\quad \leq \int_0^{\infty} \left|\int_0^z k(z,y)(1+U(z))\left(\alpha(S(B)(z,t))f(z,t)f(y,t)-\alpha(S(\tilde{B})(z,t))\tilde{f}(z,t)\tilde{f}(y,t)\right)\, dy \right| \, dz \\
			&\quad\leq \bar{k} \int_0^{\infty}\int_0^z(1+U(z))\left|\alpha(S(B)(z,t))f(z,t)f(y,t)-\alpha(S(\tilde{B})(z,t))\tilde{f}(z,t)\tilde{f}(y,t)\right| \md y \, \md z  \\
			&\quad\leq \int_0^{\infty}\int_0^z (1+U(z))\alpha(S(B)(z,t))\left| f(z,t)f(y,t)- \tilde{f}(z,t)\tilde{f}(y,t)\right|\, \md y \, \md z \\
			&\quad\phantom{\leq} + \int_0^{\infty}\int_0^z (1+U(z))\left|\alpha(S(B)(z,t))-\alpha(S(\tilde{B})(z,t))\right| \tilde{f}(z,t)\tilde{f}(y,t)\, \md y \, \md z \\
			&\quad\leq  \bar{k} \bar{\alpha} \left(\|f\|_{\mathscr{L}^1}\|f-\tilde{f}\| _{\mathscr{L}^{1}}+ \|\tilde{f}\|_{\mathscr{L}^{1}}\|f-\tilde{f}\|_{\mathscr{L}^1}\right) \\
			&\quad \phantom{\leq}+ 2\bar{k}^2 L_3 \|\tilde{f}\|_{\mathscr{L}^1}\|\tilde{f}\|_{\mathscr{L}^1} \left(\|V\|_{\mathscr{L}^{\infty}}\|g-\tilde{g}\|_{\mathscr{L}^1} + \|V-\tilde{V}\|_{\mathscr{L}^{\infty}}\|\tilde{g}\|_{\mathscr{L}^1}\right),
		\end{align*}
		where we used the Lipschitz continuity of $B \to \alpha(S(B))$ from Lemma \ref{l:lip}. Since control over $\|V-\tilde{V}\|_{\mathscr{L}^{\infty}}$ by $\|g-\tilde{g}\|_{\mathscr{L}^1}$ is ensured due to the previous estimate, we also see that $\|f-\tilde{f}\|_{\mathscr{L}^1}$ can be controlled by $\|g-\tilde{g}\|_{\mathscr{L}^1}$. Choosing the time interval $[0,T)$ small enough we can enforce
		$$
			\|f-\tilde{f}\|_{\mathscr{L}^1} \leq L \|g-\tilde{g}\|_{\mathscr{L}^1}, \quad L<1,
		$$
		which proves that $\mathcal{F}$ is a contraction concluding the proof.
	\end{proof}
	
		\section{Balanced growth path solutions}
	\label{s:bgp}
	
	In this section we discuss possible existence of balance growth path solutions, which relate to exponential growth of the overall economic production function \eqref{e:overallproductivity}. We recall that we wish to find a growth rate $\gamma >0$ to investigate the problem in the rescaled variables
	\begin{align}
		f(z,t)=e^{-\gamma t} \phi(z e^{-\gamma t}), \quad V(z,t) = e^{\gamma t} v(ze^{-\gamma t}), \quad s(z,t) = \sigma(ze^{-\gamma t}), \quad x=ze^{-\gamma t},
	\end{align} 
	where $(\phi,v, \sigma)$ are solutions to \eqref{bgp}. 
	
	Burger et al \cite{BLW} showed that BGPs exist for $k \equiv 1$ and linear utility \eqref{e:ulin} if the initial 
	CDF $F(Z,0)=\int_0^Zf_I(y) \, \md y$ \eqref{F} has a \emph{Pareto tail}.  A crucial criterion for the existence of BGP solutions was the preservation of the Pareto tail property with time, from which existence of a growth parameter $\gamma$ can be deduced. In this section, we investigate the extension of these results to the case of non-constant learning kernels. \\
	We assume that the initial CDF has a Pareto tail, that is
	\begin{align}\label{defPareto}
		\exists \, \theta, \beta > 0\,\, \text{ s.t. } \,\, \lim_{Z\to \infty} (1-F(Z,0))Z^{\frac{1}{\theta}}=\beta.
	\end{align} 
 The first Lemma shows that under these assumptions on the initial data, the Pareto tail property will be preserved in time, has the same decay $\theta$, but a different tail index $\beta$.
	\begin{lemma}\label{l:Pareto}
		Consider the Boltzmann mean-field game \eqref{bmfg} under assumption \eqref{defPareto} for the initial data and let further the learning kernel be bounded from above and below by
		$$
			\underline{k} < k(z,y) \leq \bar{k}, \quad 0 \leq \underline{k}<\bar{k},
		$$
		for all $z>y$. Then for all $t\in [0,T]$ the cumulative distribution function $F(\cdot,t)$ solving \eqref{BoEF} has a Pareto tail with the same decay rate $\theta$ as $F(\cdot,0)$.
	\end{lemma}
	\begin{proof}
		We use the reformulation from \eqref{BoEF} 
		\begin{align*}
			&\frac{\md }{\md t} F(Z,t) = - \int_Z^{\infty} \int_0^Z \alpha(s(y,t)) k(z,y)f(y,t)f(z,t)\,\md y \,\md z 
		\end{align*}
		obtained by interchanging the order of integration in the loss-term. It further follows
		\begin{align*}
			&\frac{\md }{\md t} (1-F(Z,t)) =  \int_Z^{\infty} \int_0^Z \alpha(s(y,t)) k(z,y)f(y,t)f(z,t)\,\md y \,\md z 
		\end{align*}
		and by using that $\underline{k} < k \leq \bar{k}$, we obtain that
		\begin{align*}
			(1-F(Z,t))\leq \bar{k} (1-F(Z,0))\exp\left(\int_0^t \int_0^Z \alpha(s(y,s)) f(y,s) \, \md y \, \md s \right)
		\end{align*}
		as well as 
		$
			(1-F(Z,t)) > \underline{k} (1-F(Z,0))\exp\left(\int_0^t \int_0^Z \alpha(s(y,s)) f(y,s) \, \md y \, \md s \right).
		$
	 Multiplication by $Z^{\frac{1}{\theta}}$ yields
		\begin{align*}
			Z^{\frac{1}{\theta}}(1-F(Z,t))\leq \bar{k} Z^{\frac{1}{\theta}}(1-F(Z,0))\exp\left(\int_0^t \int_0^Z \alpha(s(y,s)) f(y,s) \, \md y \, \md s \right)
		\end{align*}
		and, hence, the respective expression for the lower bound. The integral 
		$
		\int_0^Z \alpha(S(y,s)) f(y,s) \, \md y
		$
		is strictly increasing in $Z$ and bounded by $\alpha(1)$ for a fixed time $s$. Furthermore the integrability of $f(\cdot,t)$ ensures that $Z^{\frac{1}{\theta}}(1-F(Z,t))$ is monotonically decreasing for large $Z$, and therefore the limit \newline $\lim_{Z \to \infty} Z^{\frac{1}{\theta}}(1-F(Z,t))=:\beta(t)$ exists. This implies
		$$
		\underline{k} e^{t \alpha(1)}\beta < \beta(t)\leq \bar{k}e^{t \alpha(1)}\beta.
		$$
		\end{proof}

	\begin{rem}
		The bounds on $k$ are given by $\underline{k}= \tilde{\delta}<\delta$ and $\bar{k}=1$ for the polynomial kernel \eqref{e:kpoly} and by $\underline{k}=0$ and $\bar{k}=\mu$ for the exponential learning interaction \eqref{e:kexp}. For the polynomial kernel \eqref{e:kpoly} we obtain the concrete lower bound for $\beta(t)$, 
		$$	
			\delta e^{t \alpha(t)} \leq \beta(t) \leq e^{t \alpha(1)}\beta;
		$$
		for the exponential kernel \eqref{e:kexp} we only have
		$$
			0< \beta(t)\leq \mu e^{t \alpha(1)}\beta.
		$$
	\end{rem}
	
	The previous results allow us to deduce the existence of a growth parameter $\gamma$ for bounded learning kernels, which defines the $BGP$ solutions. 
	
	\begin{theorem}\label{t:growthpar}
		Let the initial CDF satisfy \eqref{defPareto} and let the learning kernel be bounded from above and below by
		$$
		\underline{k} < k(z,y) \leq \bar{k}, \quad 0 \leq \underline{k}<\bar{k},
		$$
		for all $z>y$. Then there exists a growth parameter $\gamma >0$ which can be estimated by the following conditions
		$$
		\frac{\underline{k}}{\theta} \int_0^{\infty} \alpha(\sigma(y))\phi(y) \, \md y < \gamma \leq \frac{\bar{k}}{\theta}  \int_0^{\infty} \alpha(\sigma(y))\phi(y) \, \md y.
		$$
	\end{theorem}
	
	\begin{proof}
		Lemma \ref{l:Pareto} ensures that $F(Z,t)$ has a Pareto tail for all times $t>0$. We introduce the cumulative distribution function in BGP variables by defining $F(Z,t)=\Phi(Ze^{-\gamma t})$, $X=Ze^{-\gamma t}$ and calculate for every fixed time $t \geq 0$
		$$
			\lim_{X \to \infty}X^{\frac{1}{\theta}}\left(1-\Phi(X)\right) = e^{-\frac{\gamma t}{\theta}} \beta(t) =:\tilde{\beta}(t),
		$$
		as well as
		\begin{align*}
			-\gamma X \Phi'(X) =  - \int_{Xe^{\gamma t}}^{\infty} \int_0^{Xe^{\gamma t}} \alpha(S(y,t)) k(z,y)f(y,t)f(z,t)\,\md y \,\md z . 
		\end{align*}
		If we now multiply this inequality by $X^{\frac{1}{\theta}}$ and pass to the limit $X \to \infty$, we obtain due to L'Hôpital's rule
		\begin{align*}
			\frac{\gamma}{\theta} \tilde{\beta}(t) = - \lim_{X \to \infty} \left( X^{\frac{1}{\theta}} \int_{Xe^{\gamma t}}^{\infty} \int_0^{Xe^{\gamma t}} \alpha(S(y,t)) k(z,y)f(y,t)f(z,t)\,\md y \,\md z \right).
		\end{align*}
		Since $k(z,y) \leq \bar{k} $ for $z>y$, we see that
		\begin{align*}
			&X^{\frac{1}{\theta}} \int_{Xe^{\gamma t}}^{\infty} \int_0^{Xe^{\gamma t}} \alpha(S(y,t)) k(z,y)f(y,t)f(z,t)\,\md y \,\md z 
			\, \leq \,  \bar{k} X^{\frac{1}{\theta}} (1-\phi(X)) \int_0^{Xe^{\gamma t}} \alpha(S(y,t)) f(y,t)\,\md y 
		\end{align*}
		and is therefore bounded from above by $\bar{k} \beta(t) \int_0^\infty \alpha(S(y,t)) f(y,t)\,\md y,$ as $ X \to \infty $.
		On the other hand, since $k(z,y) > \underline{k}$, we have the lower bound
		\begin{align*}
			&X^{\frac{1}{\theta}} \int_{Xe^{\gamma t}}^{\infty} \int_0^{Xe^{\gamma t}} \alpha(S(y,t)) k(z,y)f(y,t)f(z,t)\,\md y \,\md z  
			>  \underline{k} X^{\frac{1}{\theta}} (1-\Phi(X)) \int_0^{Xe^{\gamma t}} \alpha(S(y,t)) f(y,t)\,\md y,
			\end{align*}
			which gives the lower bound 
			$\underline{k} \beta(t) \int_0^{Xe^{\gamma t}} \alpha(S(y,t)) f(y,t)\,\md y$, as $X \to \infty .$
		Furthermore, we calculate 
		\begin{align*}
			&\pa_X\left[ X^{\frac{1}{\theta}} \int_{Xe^{\gamma t}}^{\infty} \int_0^{Xe^{\gamma t}} \alpha(S(y,t)) k(z,y)f(y,t)f(z,t)\,\md y \,\md z \right] \\
			&\quad =e^{\gamma t \left(1-\frac{1}{\theta}\right)} \pa_Z\left[ Z^{\frac{1}{\theta}} \int_{Z}^{\infty} \int_0^{Z} \alpha(S(y,t)) k(z,y)f(y,t)f(z,t)\,\md y \,\md z \right] \\
			&\quad = e^{\gamma t \left(1-\frac{1}{\theta}\right)} \left[\frac{1}{\theta} Z^{\frac{1}{\theta}-1} \int_{Z}^{\infty} \int_0^{Z} \alpha(S(y,t)) k(z,y)f(y,t)f(z,t)\,\md y \,\md z\right. \\
			&\quad\quad  \phantom{=e^{\gamma t \left(1-\frac{1}{\theta}\right)} } + Z^{\frac{1}{\theta}} \alpha(S(Z,t)) f(Z,t) \int_{Z}^{\infty}k(z,Z)f(z,t) \, \md z - \left. Z^{\frac{1}{\theta}} f(Z,t) \int_0^Z \alpha(S(y,t)) k(Z,y) f(y,t) \md y \right] \\
				&\quad \leq e^{\gamma t \left(1-\frac{1}{\theta}\right)} \left[ \frac{\overline{k}}{\theta Z} Z^{\frac{1}{\theta}} (1-F(Z,t)) \int_0^{Z} \alpha(S(y,t)) f(y,t) \, \md y \right.  \\
			&\quad \phantom{ \leq e^{\gamma t \left(1-\frac{1}{\theta}\right)} } + \bar{k} e^{\gamma t \left(1-\frac{1}{\theta}\right)}  Z^{\frac{1}{\theta}}(1-F(Z,t)) \alpha(S(Z,t)) f(Z,t) -\left.  \underline{k} e^{\gamma t \left(1-\frac{1}{\theta}\right)}  Z^{\frac{1}{\theta}} f(Z,t) \int_0^Z \alpha(S(y,t))f(y,t) \, \md y\right].
		\end{align*}
Both positive terms include $Z^{\frac{1}{\theta}}(1-F(Z,t))$, which (due to Lemma \ref{l:Pareto}) can be bounded by $C \beta(t)$, with a constant $C>0$ for all $Z>Z'$. This implies, together with the boundedness of $\int_0^{Z} \alpha(S(y,t)) f(y,t) \, \md y $, $\alpha(S(Z,t))$ and $\lim_{Z \to \infty} f(Z,t) = 0$, that the above expression is negative for large values of $Z \in \R_+$. We further observe that $X^{\frac{1}{\theta}} \int_{Xe^{\gamma t}}^{\infty} \int_0^{Xe^{\gamma t}} \alpha(S(y,t)) k(z,y)f(y,t)f(z,t)\,\md y \,\md z$ is non-increasing for large $X$, from which together with the bounds we are able to conclude the existence of a limit
		$$
		\lim_{X \to \infty} \int_{Xe^{\gamma t}}^{\infty} \int_0^{Xe^{\gamma t}} \alpha(S(y,t)) k(z,y)f(y,t)f(z,t)\,\md y \,\md z.
		$$
		The bounds on the growth parameter $\gamma$ follow trivially by the preceding estimates.
	\end{proof}
	
		In the case of the exponential learning kernel \eqref{e:kexp} we have $\underline{k}=0$, which implies that the growth parameter $\gamma$ can be arbitrarily close to 0. Thus, the exponential growth can be very small and the BGP solutions very close to steady states. \\

	    The existence of BGP solutions is still an open problem. We expect that the proof follows the arguments of Burger et al. \cite{BLW2} at least for linear utility functions. We leave this interesting problem in the future research.

	\section{Local mean field game model for knowledge growth}\label{s:lmfg}
	
	We finish considering BMFG models \eqref{bmfg} with non-constant learning rate $k(\cdot,\cdot)$, in particular localised kernels. We will see that this allows us to formally derive a \emph{local mean-field game model} in a suitable scaling limit.\\
	
	We assume that
	\begin{align}\label{a:klocal}
		k(z,y) = \ve^{-1}k_*\left(\frac{z-y}{\ve}\right),
	\end{align}
	where $k_*$ should be symmetric and $\ve \ll 1$ is the scaling parameter. Moreover, we assume that 
	$$
	\int_{\frac{z}{\ve}}^{\infty} k_*(x) \, \md z \in \mathcal{O}(\ve^2), \quad \text{and} \quad \int_{\frac{z}{\ve}}^{\infty} x k_*(x) \, \md z \in \mathcal{O}(\ve)
	$$
	for $\ve \to 0$. Hence individuals only learn from others with almost same knowledge level as themselves. \\
	Note that the exponential interaction rate \eqref{e:kexp} fulfils the aforesaid assumptions.  Define $\kappa = \ve^{-1}\tilde{\kappa}$ and $\mu = \ve^{-1}\tilde{\mu}$ then 
		$$
		k(z,y) = \frac{\tilde{\mu}}{\ve} e^{-\tilde{\kappa}\frac{|z-y|}{\ve}},
		$$
		where $\tilde{\kappa}, \tilde{\mu} = \mathcal{O}(1)$.\\

	We first investigate the Boltzmann equation \eqref{bmfgf} with the kernel of the form \eqref{a:klocal} in the limit $\ve \to 0$. We calculate
	\begin{align*}
		\pa_t f(z,t) &= f(z,t) \left[\ve^{-1}\int_0^z \alpha\left(S(y,t)\right) k_*\left(\frac{z-y}{\ve}\right) f(y,t) \, \md y - \alpha\left(S(z,t)\right) \ve^{-1} \int_z^{\infty} k_*\left(\frac{z-y}{\ve}\right) f(y,t) \, \md y \right] \\
		&= f(z,t) \left[\int_0^{\frac{z}{\ve}} \alpha(S(z-\ve x,t))k_*(x) f(z-\ve x,t) \, \md x - \alpha(S(z,t)) \int_0^{\infty} k_*(x) f(z+\ve x) \, \md x  \right]
	\end{align*}
	where we first performed the coordinate change $x=\frac{z-y}{\ve}$ in the gain term and $x=\frac{y-z}{\ve}$ in the loss term. By expanding $f(\cdot,t)$ and $\alpha(S(\cdot,t))$ around $z$ we further obtain
	\begin{align*}
		\pa_t f(z,t) &=   f(z,t) \left[f(z,t)\alpha(S(z,t))\int_0^{\frac{z}{\ve}} k_*(x) \, \md x - \ve \pa_zf(z,t)\alpha(S(z,t))  \int_0^{\frac{z}{\ve}}x k_*(x) \, \md x \right. \\
		&\phantom{=}\left.- f(z,t)\pa_z (\alpha(S(z,t))) \int_0^{\frac{z}{\ve}}x k_*(x) \, \md x - f(z,t) \alpha(S(z,t)) \int_0^{\infty} k_*(x) \, \md x  \right. \\
		&\phantom{=}\left. - \ve \alpha(S(z,t))\pa_zf(z,t)\int_0^{\infty} x k_*(x) \, \md x \right]  + \mathcal{O} (\ve^2) \\
		&=-\ve \lambda f(z,t) \bigg[2\pa_zf(z,t)\alpha(S(z,t)) + f(z,t) \pa_z(\alpha(S(z,t)))\bigg] + \mathcal{O} (\ve^2)  \\
		&= - \ve \lambda \pa_z\bigg[f^2(z,t)\alpha(S(z,t))\bigg]+ \mathcal{O} (\ve^2) \,,
	\end{align*}
	where we defined 
	\begin{align}\label{lambda}
		\lambda := \int_{0}^\infty x k_*(x) \, \md x.
	\end{align}
	Using \eqref{a:klocal} in the Hamilton-Jacobi-Bellman equation \eqref{bmfgV} yields
	\begin{align*}
		\pa_t V(z,t) - r V(z,t) = - \max_{s \in \mathcal{S}}{\left[U((1-s)z)+\alpha(s) \ve^{-1}\int_z^{\infty}\left(V(y,t)-V(z,t)\right)f(y,t)k_*\left(\frac{z-y}{\ve}\right)\, \md y \right]}.
	\end{align*}
	Next we expand the integral on the right-hand-side (after the coordinate chage $x=\frac{y-z}{\ve}$): 
	\begin{align*}
		\ve^{-1}\int_z^{\infty}&\left(V(y,t)-V(z,t)\right)f(y,t)k_*\left(\frac{z-y}{\ve}\right)\, \md y \\
		&= \int_0^{\infty}\left(V(z+\ve x,t)-V(z,t)\right)f(z+\ve x,t)k_*\left(x\right)\, \md x \\
		&= \ve \lambda f(z,t) \pa_zV(z,t) + \mathcal{O}(\ve^2),
	\end{align*}
	where the constant $\lambda$ is defined as in \eqref{lambda}. We rescale the interaction function 
	\begin{align}\label{e:scalealpha}
	\tilde{\alpha} = \ve \lambda \alpha,
	\end{align}
	hence individuals interact at a very high rate with others sharing the same knowledge level. We then
	obtain (omitting higher order terms and the tilde notation) the \emph{local mean field game model}
	\begin{subequations}\label{localmfg0}
	\begin{align}
		&\pa_t f(z,t) = -\pa_z\left(f^2(z,t)\alpha(s(z,t))\right), \label{e:burgers0} \\
		& \pa_tV(z,t)-rV(z,t) = -\max_{s \in \mathcal{S}}{\left[U((1-s)z)+\alpha(s(z,t)) f(z,t) \pa_zV(z,t)\right]}.
	\end{align}
	\end{subequations}
	We see that the Boltzmann equation formally converges to a Burgers' type equation. It is well know that the viscous Burgers' equation, which corresponds to the case of a constant learning function $\alpha \equiv 1$, admits travelling wave solutions. We recall that the rescaled original BMFG model, that is the system in log-variables with $k\equiv 1$, has travelling wave solutions in case of a constant learning function $\alpha \equiv 1$. \\
	
	Next we introduce the control variable 
	$$
	v(z,t) = \alpha(s(z,t))f(z,t), \quad s(z,t) = \alpha^{-1} \left(\frac{v(z,t)}{f(z,t)}\right)
	$$
	with $\mathcal{V} := \{v:\, \R_+ \to \R_+\}.$ This reformulation allows us to write \eqref{localmfg} as a generalised potential mean field game with a conservation law constraint.  In particular
	\begin{subequations}\label{localmfg}
		\begin{align}
			&\pa_tf(z,t)+\pa_z(f(z,t)v(z,t)) = 0 , \label{mfgf}\\
			&\pa_tV(z,t) - rV(z,t) = - \max_{v \in \mathcal{V}}{\left[U\left(\left(1-\alpha^{-1} \left(\frac{v(z,t)}{f(z,t)}\right)\right)z\right)+v(z,t)\pa_zV(z,t) \right]}, \label{mfgV} \\
			&f(z,0)=f_I(z), \label{mfgfI} \\
			&V(z,T)=0. \label{mfgVT}
		\end{align}
	\end{subequations}
	This system can be written as an \emph{optimal control problem} or \emph{potential mean field game}
	\begin{equation}\label{opV}
		\max_{v \in \mathcal{V}}{\int_0^T \int_0^\infty e^{-rt} w\left(\frac{v(z,t)}{f(z,t)},z\right) f(z,t) \, dz \, \md t}
	\end{equation}
	subject to 
	\begin{equation}\label{opf}
		\pa_tf(z,t) + \pa_z\left(v(z,t)f(z,t)\right) = 0,
	\end{equation}
	where due to the choice of the utility of an agent's productivity $U$, the function $w(\cdot,\cdot)$ has to fulfil the following differential equation
	\begin{align}\label{c:w}
		w(p,z)- p \pa_p w(p,z)= -U((1-\alpha^{-1}(p))z), \quad \forall \, z \in \R_+.
	\end{align}
	Note that \eqref{opV} is a generalisation of a potential mean field game. There is a well established existence theory for potential mean field games, see \cite{C2013, Benamou2017}, and we believe that some of these results can be generalised for certain utility functions (like the linear utility). However, a general existence result is an open problem that we will address in the future. 
	\begin{rem}
		Solving the maximization problem in the Hamilton-Jacobi-Bellman equation gives the following optimality condition
		$$
			v(z,t)= f(z,t) H^{-1}\left(z^{-1}f(z,t)\pa_zV(z,t)\right), 
		$$
		with $H(r(z,t)):=U'\left(\left(1-\alpha^{-1}\left(r(z,t)\right)\right)z\right)\left(\alpha^{-1}\right)'\left(r(z,t)\right)$ monotonically decreasing due to Assumptions \ref{a:alpha}, \ref{a:U}. We see that  the monotonicity of $\pa_zV(\cdot,t)$ has a direct influence on the monotonicity of the control $v(\cdot,t)$; a similar correlation that we already discussed for the non-local Hamilton-Jacobi-Bellman equation in Remark \ref{r:monB}.
	\end{rem}
	
	We conclude by discussing solutions to equation \eqref{c:w} for different utilities and interaction functions. Differentiating condition \eqref{c:w} with respect to $p$ yields
		\begin{align}\label{e:w}
			-p \pa_p^2 w(p,z) = z U'((1-\alpha^{-1}(p))z)\left(\alpha^{-1}\right)'(p),
		\end{align}
		Then the function $w$, relating to the time discounted running cost of an individual at time $t$
		$$
		\int_0^\infty e^{-rt} w\left(\frac{v(z,t)}{f(z,t)},z\right) f(z,t) \, dz,
		$$
can be computed explicitly in some cases assuming that the interaction function $\alpha$ is of the form
		\begin{align*}
			\alpha(s):=s^{\frac{1}{\eta}}, \quad \eta > 1.
		\end{align*}
		In the case of a linear utility
	 equation \eqref{e:w} is of the form
			$$
				\pa_p^2 w(p,z) = - z\eta p^{\eta-2}.
			$$
			Since $\eta \neq 1$ we find that
			$$
				w(p,z) = -\frac{z}{\eta-1} p^{\eta} +c_1(z)p+c_2(z),
			$$
			where $c_1(z),c_2(z) \in \R$ are suitable integration constants. In case of a logarithmic utility we obtain 
		$$
			\pa_p^2 w(p,z) =- \eta \frac{p^{\eta-2}}{1-p^\eta}.
			$$
				For $\eta=2$ it simplifies to 
				$$
				\pa_p^2 w(p,z)=-\eta \frac{1}{1-p^{2}}.
				$$
				Integration gives 
				$$
				\pa_pw(p,z) = -\frac{\eta}{2} \left(\ln{(p+1)}- \ln{(|p-1|)}\right) + c_1,
				$$
				from which we can conclude
				$$
				w(p,z) = -\eta (p+1) \ln{(p+1)}+(p-1) \ln{(|p-1|)}+c_1(z)p+c_2(z), 
				$$
				where $c_1(z), c_2(z) \in \R$ again describe integration constants. Note that $w$ is not defined for $p \equiv 1$.\\
				We can not give a closed form solution for $w$ in case of an isoelastic utility. 

	\section{Numerical simulations}\label{s:numerics}
	
	In this section we present an iterative solver for \eqref{bmfg} as well as \eqref{localmfg} and
	computational results supporting and exemplifying the presented	analysis.\\

	Consider the computational domain $\mathcal{I} = [0, 10]$ and a fixed time horizon $\mathcal{T} = [0,25]$; both split into intervals of size $\Delta z = 0.01$ and $\Delta t = 0.01$ respectively. We will use superscripts to refer to the discrete in time functions, for example $\mathbf{f}^n = f(z,t^n)$ where $t^n = n \Delta t$.\\
		We choose an initial distribution of agents, which has a Pareto tail
	\begin{align*}
	    f_I(z) = \frac{\beta}{(z+1)^\beta}
	\end{align*}
	and set the interaction function $\alpha(s) = 2\sqrt{s}$. Furthermore we set the temporal discount factor to $r=0.05$. We will choose these initial conditions, interaction function and discount factor for all computational experiments if not stated otherwise. \\
	
	\noindent The numerical solver is based on the following fixed point iteration:
	\begin{enumerate}[nosep]
		\item \textit{Explicit in time finite difference discretisation of the
			Boltzmann equation \eqref{bmfgf}:}  Given $\mathbf{A}^n	=\alpha(\mathbf{S}^n)$ for $n=0, \ldots N$, starting with $\mathbf{S}^0 \equiv 1$, solve the Boltzmann
		equation forward in time:
		\begin{align*}
			\mathbf{f}^{n+1} = \mathbf{f}^n + \Delta t \, \mathbf{f}^n  \int_0^z \mathbf{A}^n k(z,y)
			\mathbf{f}^n \md y - \Delta t\, \mathbf{A}^n \mathbf{f}^n(y) \int_z^{\bar{z}} k(y,z)
			\mathbf{f}^n \md y. 
		\end{align*}
		We calculate the integrals on the right-hand-side using the trapezoidal rule.
		\item \textit{Policy iteration to solve the HJB equation \eqref{bmfgV}:} Let
		$\mathbf{V}^n$ and $\mathbf{S}^n$ denote the time-discrete
		solution to \eqref{bmfgV} at time $t^n$ and let $\mathbf{A}^n :=
		\alpha(\mathbf{S}^n)$.
		\begin{enumerate}[nosep]
			\item Given the terminal condition $\mathbf{V}^{N}\equiv 0$ and $\mathbf{A}^{n}$, $\mathbf{S}^{n}$
			and $\mathbf{f}^{n}$ for every $n=0,1 \ldots ,N$ solve the HJB backward in time 
			\begin{align*}
				\frac{\mathbf{V}^{n+1}-\mathbf{V}^n}{\Delta t} - r \mathbf{V}^{n} =
				-(1-\mathbf{S}^{n} ) z  - \mathbf{A}^{n} \int_z^{\bar{z}}
				(\mathbf{V}^{n}(y)-\mathbf{V}^{n}(z) )\mathbf{f}^{n}(y) k(y,z) \md y
			\end{align*}
			or equivalently
			\begin{align*}
				-\left(1+r\Delta t\right)\mathbf{V}^{n} + \Delta t \, \mathbf{A}^{n} \int_z^{\bar{z}}
				(\mathbf{V}^{n}(y)-\mathbf{V}^{n}(z) )\mathbf{f}^{n}(y) k(y,z) \md y =
				-\mathbf{V}^{n+1} - \Delta t(1-\mathbf{S}^{n} ) z.
			\end{align*}
			\item Given $\mathbf{V}^n$ and $\mathbf{f}^n$ for $n=0, \ldots N$,
			update $\mathbf{S}^n$ by computing \newline
			$\mathbf{B}^n = \int_z^{\bar{z}}
			(\mathbf{V}^{n}(y)-\mathbf{V}^n(z))\mathbf{f}^n(y) k(y,z)dy$ for
			all time $t^n$, $n=0, \dots N$ and calculating (elementwise) for the linear utility:
			\begin{itemize}[nosep]
				\item $\mathbf{S}^n=0$ for $\mathbf{B}^n \leq 0$.
				\item $\mathbf{S}^n =1$ if $\mathbf{B}^n \mathbf{A}(1) \leq z$
				\item $\alpha'(\mathbf{S}^n) =\frac{z}{\mathbf{B}^n}$.
			\end{itemize}
		\end{enumerate}
	\end{enumerate}
	Note that the update of $S$ in (8b) has to be adapted in case of the isoelastic and logarithmic utility.

	\subsection{Polynomial learning kernel}
	
		\paragraph{Linear utility:}

	We start by illustrating the behaviour of solutions in case of a polynomial learning kernel $k$ as defined in \eqref{e:kpoly} with parameters 
	\begin{align*}
	    \delta =0.5 \text{ and } \kappa = 1
	\end{align*}
	and a linear utility $U$ given by \eqref{e:ulin}. Note that these parameter satisfy condition \eqref{kappa_delta}; we therefore expect monotonicity of $V$ and $S$. Figure \ref{f:poly_nv} shows the expected behaviour of the learning function $S$ as well as the value function $V$ at different times. We see that agent choose to spend their entire time on learning, therefore $S\equiv 1$, up to a certain knowledge level $\tilde{z}$. The value of $\tilde{z}$ decreases as time increases.
	\begin{figure}[h!]
		\centering
		\subfigure[$V$]{\includegraphics[scale=0.4]{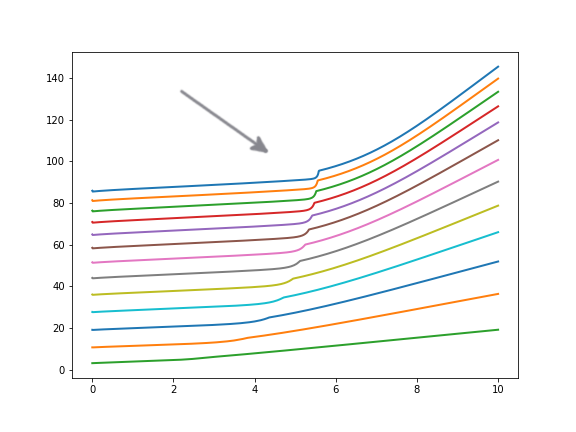}}
		\hspace*{1em}
		\subfigure[$S$]{\includegraphics[scale=0.4]{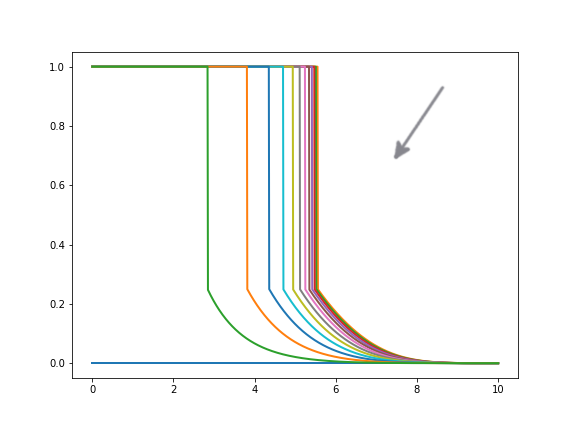}}
		\caption{Solution to \eqref{bmfg} with linear utility and a polynomial learning kernel \eqref{e:kpoly} with $\delta = 0.5$ and $\kappa = 1$, which satisfies \eqref{kappa_delta}. The direction of the arrow indicates the increase of time.}
		\label{f:poly_nv}
	\end{figure}
	
	\noindent The numerically observed monotonicity behaviour does not change even if we choose parameters violating condition \eqref{kappa_delta}. For example, if we set
	\begin{align*}
	    \delta = 0.1 \text{ and } \kappa = 3
	\end{align*}
	solutions are still monotone, see Figure \ref{f:poly_v}. 
	We observe that $S$ is still non-increasing, yet the incentive to learn is lower than in the previous example, as the rate of learning is smaller and decay faster.
		\begin{figure}[h!]
		\centering
		\subfigure[$V$]{\includegraphics[scale=0.4]{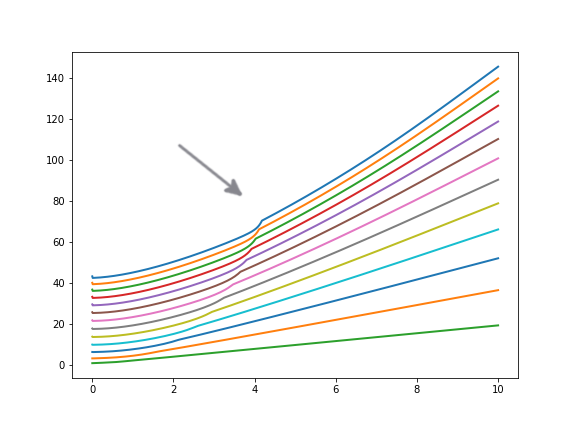}}
		\hspace*{1em}
		\subfigure[$S$]{\includegraphics[scale=0.4]{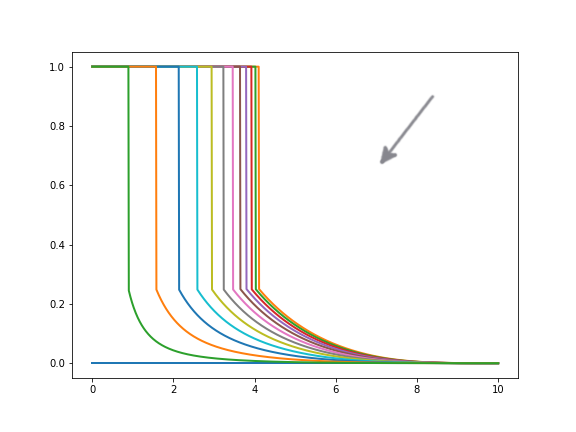}}
		\caption{Solution to \eqref{bmfg} with linear utility and a polynomial learning kernel \eqref{e:kpoly} with $\delta = 0.1$ and $\kappa = 3$, which violates \eqref{kappa_delta}. The direction of the arrow indicates the increase of time.}
		\label{f:poly_v}
	\end{figure}

	\paragraph{Isoelastic utility} Next we use the isoelastic utility \eqref{e:ucrra} with $\zeta = 0.5$, in particular 
	$$U(p) = \sqrt{p}.
$$
Figure \ref{f:iso} shows the value function $V$ and the optimally allocated fraction of time $S$ for a polynomial interaction kernel with $\delta = 0.5$ and $k=1$. We observe that $S$ only takes the value $1$ at $z=0$, and decreases for all $z>0$. This can be explained by the fact that we never satisfy the conditions stated in the first bullet point in the proof of Corollary \ref{c:Smon}, in particular $U'(0) < \infty$ holds for the linear utility only.

	\begin{figure}[h!]
		\centering
		\subfigure[$V$]{\includegraphics[scale=0.4]{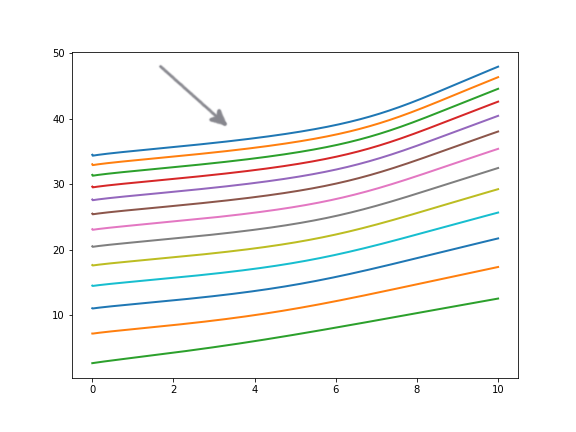}}
		\hspace*{1em}
		\subfigure[$S$]{\includegraphics[scale=0.4]{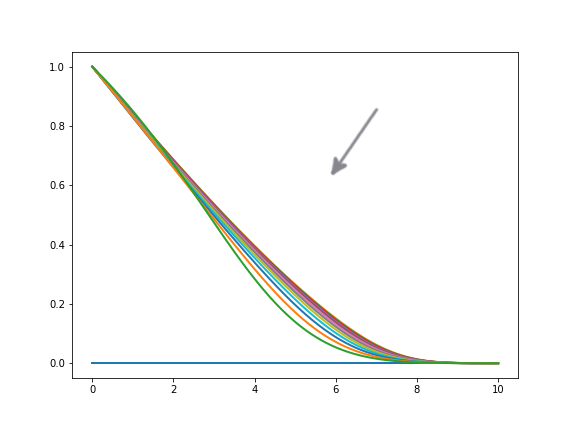}}
		\caption{Solution to \eqref{bmfg} with square root utility and a polynomial learning kernel \eqref{e:kpoly} with $\delta = 0.5$ and $\kappa = 1$. The direction of the arrow indicates the increase of time.}
		\label{f:iso}
	\end{figure}
	
	\paragraph{Logarithmic utility} Lucas and Moll \cite{ML} reported non-monotonic behaviour of solutions to \eqref{bmfg} for logarithmic utilities of the form \eqref{e:ulog}. We recall that most of the presented analytical results exclude log utilities. In the computations we consider the following regularised version of \eqref{OP}:
	\begin{align}
	    \max_{s \in S} \left[\log((1-s)(z+\varepsilon) + \alpha(s) B\right]
	\end{align}
	with $\varepsilon \ll 1$. Then the first-order optimality condition for $S$ and $\alpha = \alpha_0 \sqrt{s}$ reads as
	\begin{align*}
	    \frac{-1}{(1-s)} + \frac{\alpha_0}{2\sqrt{s}} B = 0.
	\end{align*}
	Hence the maximiser is independent of the regularisation parameter $\varepsilon$. Figure shows the utility $V$ as well as the optimal allocated time. We observe that the value function does indeed take negative values (as stated in Lemma \ref{l:Vposln}), and that we cannot expect monotonicity of solutions.
	\begin{figure}[h!]
		\centering
		\subfigure[$V$]{\includegraphics[scale=0.4]{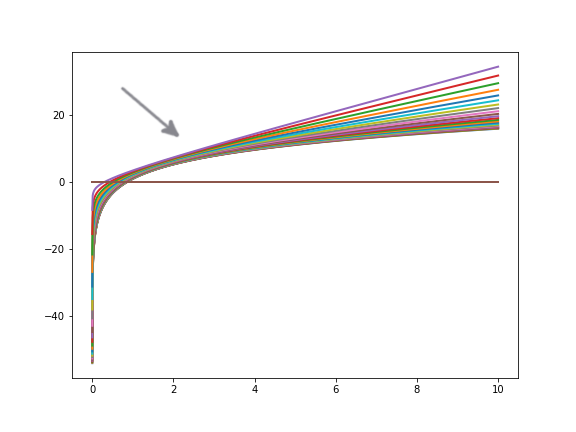}}
		\hspace*{1em}
		\subfigure[$S$]{\includegraphics[scale=0.4]{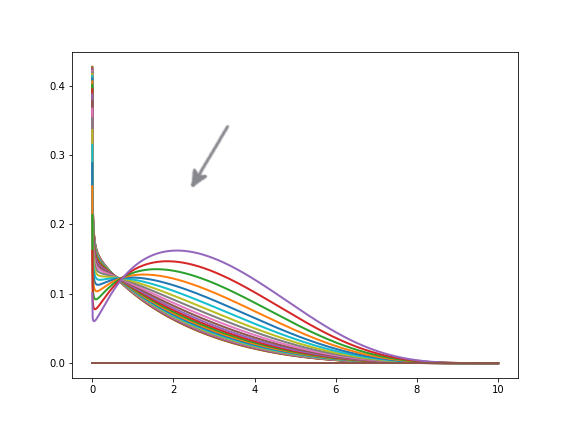}}
		\caption{Solution to \eqref{bmfg} with logarithmic utility and a polynomial learning kernel \eqref{e:kpoly} with $\delta = 0.05$ and $\kappa = 1.5$. The direction of the arrow indicates the increase of time. }
		\label{f:log_v}
	\end{figure}

	\subsection{Exponential learning kernel}
	We continue by considering exponential learning kernels of the form \eqref{e:kexp} with $\kappa = 1$. Although we were not able to show monotonicity of solutions on $\mathcal{I}$, we observe a very similar behaviour as for the polynomial learning kernel in Figure \ref{f:bmfgexp}. The faster decay of the learning rate leads to a faster decay of $S$; however, the qualitative behaviour of solutions is very similar.
		\begin{figure}[h!]
		\centering
		\subfigure[$V$]{\includegraphics[scale=0.4]{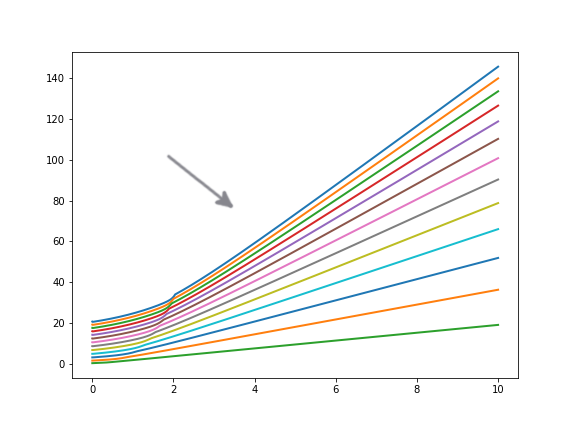}}
		\hspace*{1em}
		\subfigure[$S$]{\includegraphics[scale=0.4]{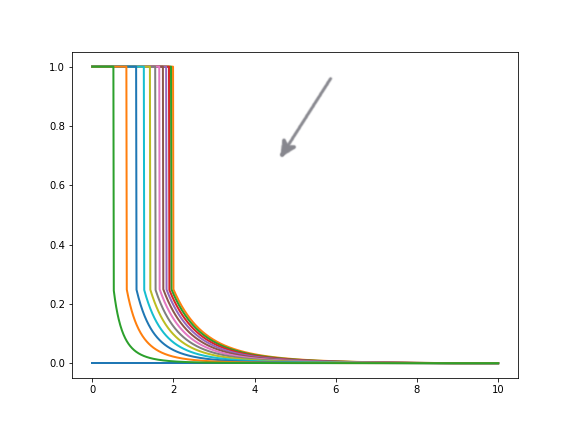}}
		\caption{Solution to \eqref{bmfg} with linear utility and a exponential learning kernel \eqref{e:kexp} with $\kappa = 1$. The direction of the arrow indicates the increase of time.}
		\label{f:bmfgexp}
	\end{figure}

	\subsection{Local mean-field games}
	We conclude by showing first computational results, to illustrate solutions to \eqref{localmfg0}. In doing so we replace Step 1 in the fixed point algorithm with a solver for a Burger's equation with general mobility. In particular 
	\begin{enumerate}
	\item \textit{Explicit in time finite difference discretisation of Burger's equation \eqref{e:burgers0}}, see \cite{T2001}: Given $\mathbf{A}^n = \alpha(\mathbf{S}^n)$ for $n=0, \ldots N$ solve Burger's equation forward in time
		\begin{align*}
			\mathbf{f}^{n+1}_j = \mathbf{f}^n_j + \lambda \left(\bar{\mathbf{A}}^n_{j+\frac{1}{2}} \mathbf{J}_{j+\frac{1}{2}} - \bar{\mathbf{A}}^n_{j-\frac{1}{2}} \mathbf{J}_{j - \frac{1}{2}}  \right)
		\end{align*}
	where $\lambda = \frac{\Delta t}{\Delta x}$, $\mathbf{J}_{j+\frac{1}{2}}$ is a Godunov flux and $\bar{\mathbf{A}}$ the average of $\mathbf{A}$ over the cell $I_j = [x_j - \frac{\Delta x}{2}, x_j + \frac{\Delta x}{2}]$. Step 2 of the fixed point iteration is unchanged, except for the definition of $B$, which is given by
	\begin{align*}
	    \mathbf{B}^n = \mathbf{D}(\mathbf{V}^n) \mathbf{f}^2,
	\end{align*}
	where $\mathbf{D}$ is the discrete gradient operator.\\
	Figure \ref{f:lmfg} shows the solution to the local MFG model, using the same parameters as in the previous examples. We observe a very similar behaviour to the BMFG - such as monotonicity of solutions - matter of future research.
	\end{enumerate}
		\begin{figure}[h!]
		\centering
		\subfigure[$V$]{\includegraphics[scale=0.4]{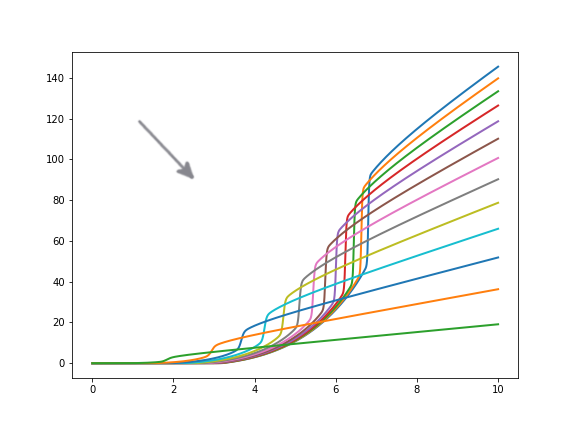}}
		\hspace*{1em}
		\subfigure[$S$]{\includegraphics[scale=0.4]{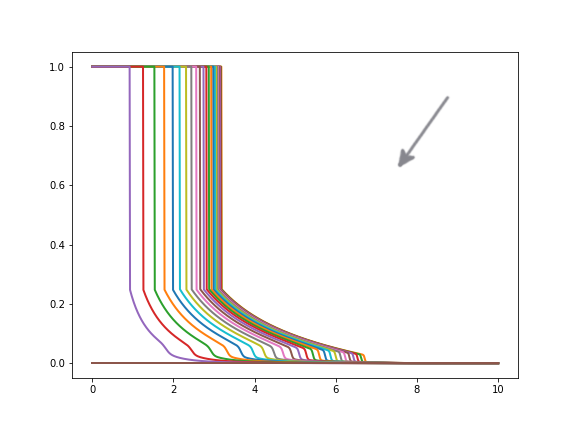}}
		\caption{Solution to \eqref{localmfg0} with linear utility. The direction of the arrow indicates the increase of time. }
		\label{f:lmfg}
	\end{figure}
	\newpage 
	\textbf{Acknowledgements:}
	\textit{M.B. acknowledges support from DESY (Hamburg, Germany), a member of the Helmholtz Association HGF, and the German Science Foundation (DFG) through CRC TR 154, subproject C06. L.K. received funding by a grant from the FORMAL team at ISCD - Sorbonne Université.}
	
	\printbibliography
	
\end{document}